\documentclass[reqno,openany,dvipdfmx]{amsart}
\usepackage{amsthm,amsmath,amssymb,amsthm}
\usepackage[left=3.5cm, right=3.5cm, includefoot]{geometry}
\usepackage[all]{xy}
\usepackage{color}
\usepackage{appendix}
\usepackage{tikz}

\theoremstyle{definition}
\newtheorem{theorem}{Theorem}[section]
\newtheorem{proposition}{Proposition}[section]
\newtheorem{lemma}{Lemma}[section]
\newtheorem{corollary}{Corollary}[section]

\newtheorem{example}{Example}[section]
\newtheorem{definition}{Definition}[section]
\newtheorem{remark}{Remark}[section]

\numberwithin{equation}{section}

\begin{document}
\title{Markov semigroups on unitary duals generated by quantized characters}
\author[R. Sato]{Ryosuke SATO}
\address{Graduate School of Mathematics, Nagoya University, Chikusaku, Nagoya 464-8602, Japan}
\email{d19001r@math.nagoya-u.ac.jp}
\maketitle

\begin{abstract}
In this paper, we study Markov dynamics on unitary duals of compact quantum groups. We construct such dynamics from characters of quantum groups. Then we show that the dynamics have generators, and we give an explicit formula of the generators using the representation theory. Moreover, we construct Markov dynamics on the unitary dual of an inductive limit of compact quantum groups.
\end{abstract}

\allowdisplaybreaks{
\section{Introduction}
The purpose of this paper can be explained in two ways. The first is to investigate the relationship between stochastic processes on unitary duals of (quantum) groups and representation theory. The second is to study dual objects of stochastic processes on (quantum) groups.

If a continuous function on a topological group $G$ is positive-definite, central, and normalized, it is called a character of $G$. Then the set of characters is convex, and the set $\widehat G$ of extreme points is called the unitary dual of $G$. If $G$ is a compact group, Fourier analysis gives a one-to-one correspondence between the set of characters of $G$ and the set of probability measures on $\widehat G$. Moreover, we can extend this correspondence to inductive limit groups $G_\infty=\varinjlim_NG_N$ of compact groups $G_N$. For $G_\infty$ there also exists a one-to-one correspondence between the set of probability measures on $\widehat{G_\infty}$ and the set of so-called \emph{central} probability measures on the \emph{chains} of elements in the $\widehat{G_N}$. In the case of the infinite-dimensional unitary group $U(\infty)=\varinjlim_NU(N)$, this is a one-to-one correspondence between the set of probability measures on $\widehat{U(\infty)}$ and the set of \emph{central} probability measures on the \emph{Gelfand--Tsetlin patterns}. Using these correspondences, we can study the characters of $U(\infty)$ in probabilistic ways. See \cite{Boyer83}, \cite{BorodinOlshanski}, \cite{OkounkovOlshanski} \cite{Olshanski03}, \cite{StraVoic:book}, \cite{VK82}, ... etc. Moreover, we can extend such a probabilistic approach for characters to the inductive limits of compact \emph{quantum} groups. See \cite{Sato1}, \cite{Sato3}. If we deal with the quantum unitary groups $U_q(N)$, then there exists a one-to-one correspondence between the set of characters of $U_q(\infty)$ and the set of \emph{$q^2$-central} probability measures on the Gelfand--Tsetlin patterns. See also \cite{Gorin12}.

It is known that Gelfand--Tsetlin patterns involve several statistical mechanical interpretations. See \cite{BC}, \cite{BP} and their references. Thus, a time-evolution of probability measures on the Gelfand--Tsetlin patterns (i.e., on $\widehat{U(\infty)}$) gives a time-evolution of such models in statistical mechanics. The purpose of this paper is giving a general idea to construct Markov dynamics on $\widehat{G_\infty}$ based on the representation theory of the inductive limit $G_\infty=\varinjlim_NG_N$ of compact quantum groups $G_N$.

For a compact quantum group $G$, the space $A(G)$ of matrix coefficients of finite-dimensional corepresentations of $G$ has a Hopf-algebra structure. There are several studies about Markov dynamics on $A(G)$. See \cite{CFK}, \cite{Franz06}, ... etc. The basic idea in those papers is constructing Markov dynamics using the comultiplication of $A(G)$ and \emph{conditionally positive} hermitian linear functionals on $A(G)$ (i.e., linear functionals such that positive on the kernel of the counit). In this paper, we will construct dual objects of such Markov dynamics.

Kac algebras are von Neumann algebras with an additional structure like Hopf algebras. Then they give one of the formulations of the duality of locally compact groups (see \cite{ES}). For a locally compact quantum group $G$, \emph{Woronowicz algebras} play a similar role in their duality theory (see \cite{MasudaNakagami}). In the duality theory, one of the Woronowicz algebras is given by a von Neumann algebra like $L^\infty(G)$, and another one is given by the group von Neumann algebra $W^*(G)$ that is generated by the left regular representation of $G$. Therefore, dual objects of Markov dynamics on $A(G) (\subset L^\infty(G))$ should be Markov dynamics on $W^*(G)$. In this paper, we construct (quantum) Markov dynamics on $W^*(G)$ using the comultiplication of $W^*(G)$ and conditionally positive hermitian linear functionals on $W^*(G)$. Moreover, some our dynamics induce Markov dynamics on $\widehat G$ by an isomorphism between the space $\ell^\infty(\widehat G)$ and the center $Z(W^*(G))$ of $W^*(G)$.

The first main result in the paper is Theorem \ref{thm:jun}. This theorem gives a construction of \emph{continuous-time} Markov dynamics on the unitary dual $\widehat G$ from characters of a compact quantum group $G$. Then we give an explicit formula of the generator of the dynamics. Moreover, for certain characters of $U(N)$ and $U_q(N)$, we show a determinantal formula of the generators in Theorem \ref{thm:kuma}, \ref{thm:q-pubg}. The construction here can be extended to inductive limit quantum groups of compact quantum groups. See \ref{sec:ind}. 

The second main result is Theorem \ref{thm:d-dynamics}. In Section \ref{sec:discrete-time}, we investigate \emph{discrete-time} dynamics on the set of Gelfand--Tsetlin patterns. Our construction is also based on the representation theory. Indeed, it is a generalization of algebraic construction of Markov dynamics due to Kuan \cite{Kuan18}. Then our dynamics are given by \emph{Toeplitz-like} transition probabilities due to Borodin and Ferrari \cite{BF}. Moreover, it is an algebraic construction of the dynamics in \cite{BG}.

The organization of this paper is the following: Section \ref{sec:notations} serves to prepare necessary notations. We review the basic facts about Woronowicz algebras in Section \ref{sec:wor}. Through the end of the section and Section \ref{sec:semigroups}, we construct Markov dynamics on the center of general Woronowicz algebra with a counit. In Section \ref{sec:cqg}, we apply the construction of Markov dynamics when a Woronowicz algebra is a group von Neumann algebra of compact quantum group. Then we prove Theorem \ref{thm:jun}. In Section \ref{sec:unitary} and \ref{sec:quantum_unitary}, we study the concrete examples from the unitary groups $U(N)$ and the quantum unitary groups $U_q(N)$. In Section \ref{sec:ind}, we construct Markov dynamics on unitary duals of inductive limit quantum groups of compact quantum groups. In Section \ref{sec:GT}, we show that the dynamics on $\widehat{U(\infty)}$ is a Feller semigroup. We also discuss discrete-time dynamics on Gelfand--Tsetlin patterns in Section \ref{sec:discrete-time}. In Appendix \ref{app:FA}, we review Fourier analysis of compact quantum groups to fix the necessary notations.

\section{Notations}\label{sec:notations}
Here we prepare necessary notations in the paper. 

Let $M$ be a $W^*$-algebra. We denote by $M_*, M_+$ the predual of $M$ and the set of positive elements in $M$, respectively. For a normal semi-finite weight $\varphi\colon M_+\to [0, \infty]$ we define the left ideal $\mathfrak{n}_\varphi:=\{x\in M\mid \varphi(x^*x)<\infty\}$. The triple $(\pi_\varphi, H_\varphi, \eta_\varphi)$ is called the GNS(Gelfand--Naimark--Segal)-triple associated with $\varphi$ if $(\pi_\varphi, H_\varphi)$ is a non-degenerate $*$-representation of $M$ and $\eta_\varphi\colon\mathfrak{n}_\varphi\to H_\varphi$ is a linear map with a dense range such that $\pi_\varphi(x)\eta_\varphi(y)=\eta_\varphi(xy)$ for any $x\in M$ and $y\in\mathfrak{n}_\varphi$, and $\varphi(y^*x)=\langle\eta_\varphi(x),\eta_\varphi(y)\rangle$ for any $x, y\in\mathfrak{n}_\varphi$. 

For a family $\{A_\alpha\}_{\alpha\in I}$ of $C^*$-algebras we define $\bigoplus_{\alpha\in I}A_\alpha\subset c_0\mathchar`-\bigoplus_{\alpha\in I}A_\alpha\subset\ell^\infty\mathchar`-\bigoplus_{\alpha\in I}A_\alpha$ by
\[\bigoplus_{\alpha\in I}A_\alpha:=\left\{(x_\alpha)_{\alpha\in I}\in\prod_{\alpha\in I}A_\alpha\, \middle|\, x_\alpha=0\text{ without finitely many }\alpha\in I\right\},\]
\[c_0\mathchar`-\bigoplus_{\alpha\in I}A_\alpha:=\left\{(x_\alpha)_{\alpha\in I}\in\prod_{\alpha\in I}A_\alpha\, \middle|\,\lim_{\alpha\in I}\|x_\alpha\|_{A_\alpha}=0\right\},\]
\[\ell^\infty\mathchar`-\bigoplus_{\alpha\in I}A_\alpha:=\left\{(x_\alpha)_{\alpha\in I}\in\prod_{\alpha\in I}A_\alpha\, \middle|\,\sup_{\alpha\in I}\|x_\alpha\|_{A_\alpha}<\infty\right\},\]
where $\lim_{\alpha\in I}\|x_\alpha\|=0$ if for any $\epsilon>0$ there exists a finite subset $I_\epsilon\subset I$ such that $\|x_\alpha\|<\epsilon$ for any $\alpha\in I\backslash I_\epsilon$. Then $c_0\mathchar`-\bigoplus_{\alpha\in I}A_\alpha$ is a $C^*$-algebra and $\ell^\infty\mathchar`-\bigoplus_{\alpha\in I}A_\alpha$ is a $W^*$-algebra.

For a vector space $V$ we define the flip map $\sigma_V\colon V\otimes V\to V\otimes V$ by $\sigma_M(\xi\otimes\eta)=\eta\otimes\xi$ for any $\xi, \eta\in V$. We will use the leg numbering notations. Namely, for a linear map $T\colon V\otimes V\to V\otimes V$ we define 
\[T_{12}:=T\otimes\mathrm{id}_V,\quad T_{23}:=\mathrm{id}_V\otimes T,\quad T_{13}:=\sigma_{V_{23}}T_{12}\sigma_{V_{23}}\] 
as linear maps on $V\otimes V\otimes V$.

For $\lambda=(\lambda_1\geq\cdots\geq\lambda_N)\in\mathbb{Z}^N$ the Schur (Laurent) polynomial $s_\lambda(z_1,\dots,z_N)$ is defined as
\[s_\lambda(z_1,\dots,z_N)=\frac{\det[z_i^{\lambda_j+N-j}]_{i, j=1}^N}{V_N(z_1,\dots, z_N)},\]
where $V_N(z_1,\dots, z_N):=\prod_{1\leq i<j\leq N}(z_i-z_j)$. Remark that $s_\lambda(z_1,\dots,z_N)$ becomes a Laurent polynomial and it is a polynomial if $\lambda_N\geq0$. 

For a topological space $X$ we denote by $\mathcal{B}(X)$ and $\mathcal{M}_p(X)$ its Borel $\sigma$-algebra and the set of Borel probability measures on $X$, respectively. For two topological spaces $X, Y$ a \emph{Markov kernel} $\Lambda$ from $Y$ to $X$ is a function $\Lambda\colon Y\times \mathcal{B}(X)\to[0,1]$ such that $\Lambda(y, \,\cdot\,)$ is in $\mathcal{M}_p(X)$ for any $y\in Y$ and $\Lambda(\,\cdot\,, A)$ is measurable on $Y$ for any $A\in\mathcal{B}(X)$. Remark that $\Lambda$ induces the mapping $m\in\mathcal{M}_p(Y)\mapsto m\Lambda\in \mathcal{M}_p(X)$ by $m\Lambda(A)=\int_Ym(dy)\Lambda(y,A)$ for any $A\in \mathcal{B}(X)$.

\section{Woronowicz algebras}\label{sec:wor}
Here we summarize necessary notions related to Woronowicz algebras due to Masuda and Nakagami \cite{MasudaNakagami}. A tuple $\mathcal{M}=(M, \delta, R, \tau, h)$ is a \emph{Woronowicz algebra} if
\begin{itemize}
\item $M$ is a $W^*$-algebra, 
\item $\delta\colon M\to M\bar\otimes M$ is a \emph{comultiplication}, i.e., it is a faithful normal unital $*$-homomorphism such that $(\delta\otimes\mathrm{id})\delta=(\mathrm{id}\otimes\delta)\delta$,
\item $R\colon M\to M$ is a \emph{unitary antipode}, i.e., it is a involutive $*$-anti-automorphism such that $\delta R=\sigma_M(R\otimes R)\delta$,
\item $\tau=(\tau_t)_{t\in\mathbb{R}}$ is a \emph{deformation automorphism group}, i.e., it is a one-parameter automorphism group on $M$ such that $(\tau_t\otimes\tau_t)\delta=\delta\tau_t$ and $R\tau_t=\tau_t R$ for any $t\in\mathbb{R}$,
\item $h$ is a \emph{left invariant Haar weight}, i.e., it is a $\tau$-invariant faithful normal semi-finite weight on $M$ such that $(\omega\otimes h)\delta=\omega(1)h$ for any $\omega\in M_*$, the modular automorphism groups of $h$ and $h R$ commute, and
$(\omega\otimes h)((1\otimes y^*)\delta(x))=(\omega\kappa\otimes h)(\delta(y^*)(1\otimes x))$
for any $x, y\in\mathfrak{n}_h$ and $\omega\in (M_*)_\tau$, where $\kappa:=\tau_{-\mathrm{i}/2}R$, and $(M_*)_\tau$ is the set of analytic elements with respect to the $\mathbb{R}$-action on $M_*$ given by $\omega\in M_*\mapsto \omega\tau_t\in M_*$.
\end{itemize}

\begin{remark}
If we do not assume the existence of Haar weight and assume the existence of $\sigma$-weakly dense $C^*$-algebra $A$ preserved by $\tau$, then $(M, A, \delta, R, \tau)$ is called a \emph{quantum group $W^*$-algebra}. We introduced it in the study of inductive limits of compact quantum groups in \cite{Sato3}. See also \cite{Yamagami}.
\end{remark}

Throughout the paper, we assume that $h$ is state, i.e., $h(1)=1$. Then $h$ is also right-invariant, i.e., $(h\otimes\omega)\delta=\omega(1)h$ for any $\omega\in M_*$. Moreover, $h$ automatically satisfies that 
$(h\otimes\omega)(\delta(y^*)x\otimes1)=(h\otimes\omega\kappa)(y^*\otimes1\delta(x))$ for any $x, y\in M$ and $\omega\in(M_*)_\tau$. See \cite{MasudaNakagami}.

Let $(\pi_h, H_h, \eta_h)$ be the GNS-triple associated with $h$. Then we define the so-called \emph{multiplicative unitary} $V\colon H_h\otimes H_h\to H_h\otimes H_h$ by $V\eta_h(x)\otimes \xi=(\pi_h\otimes\pi_h)(\delta(x))\eta_h(1)\otimes\xi$ for any $x\in M$ and $\xi\in H_h$. It is known that $V\in B(H_h)\bar\otimes M$. 

We define a multiplication on $M_*$ by $\omega_1*\omega_2:=(\omega_1\otimes\omega_2)\delta$ for any $\omega_1, \omega_2 \in M_*$. Then $M_*$ becomes a Banach algebra, and $(M_*)_\tau$ is a Banach subalgebra of $M_*$ with an involution defined by $\omega^\dagger(x):=\overline{\omega(\kappa(x)^*)}$ for $x\in M$ and $\omega\in(M_*)_\tau$. We obtain a representation $(\hat\pi_h, H_h)$ of $M_*$ such that $\hat\pi_h(\omega):=(\mathrm{id}\otimes\omega)(V)$ for any $\omega\in M_*$. Then the restriction to $(M_*)_\tau$ is a $*$-representation, i.e., $\hat\pi_h(\omega^\dagger)=\hat\pi_h(\omega)^*$ for any $\omega\in (M_*)_\tau$. We denote by $\hat M$ the von Neumann algebra generated by $\hat\pi_h(M_*)$. Then $V\in \hat M\bar\otimes M$. See \cite[Section 2]{MasudaNakagami} for more details.

It is known that $\hat M$ also has a Woronowicz algebra structure by the following way: We define a comultiplication $\hat\delta\colon \hat M\to \hat M\bar\otimes \hat M$, a unitary antipode $\hat R\colon \hat M\to \hat M$, and a deformation automorphism group $\hat\tau=(\hat\tau_t)_{t\in\mathbb{R}}$ by
\[\hat\delta(x)=V^*(1\otimes x)V,\quad \hat R(\hat\pi_h(\omega))=\hat\pi_h(\omega R),\quad \hat\tau_t(\hat\pi_h(\omega))=\hat\pi_h(\omega\tau_{-t})\]
for any $x\in\hat M$ and $\omega\in M_*$. Moreover, we obtain a left-invariant Haar weight $\hat h$ on $\hat M$ using the theory of Hilbert algebras. See \cite[Section 3]{MasudaNakagami} for more details. Then we call $\hat{\mathcal{M}}:=(\hat M, \hat\delta, \hat R, \hat\tau, \hat h)$ the dual Woronowicz algebra of $\mathcal{M}$.

For a $W^*$-algebra $N$ a \emph{right action} of $\mathcal{M}$ on $N$ is a faithful normal unital $*$-homomorphism $\alpha\colon N\to N\bar\otimes M$ satisfying that $(\alpha\otimes\mathrm{id})\alpha=(\mathrm{id}\otimes\delta)\alpha$. Then we define the fixed-point algebra 
\[N^\alpha:=\{x\in N\mid \alpha(x)=x\otimes1\}.\]
A normal linear functional $\varphi\in N_*$ is \emph{$\alpha$-invariant} if 
$(\varphi\otimes\mathrm{id})(\alpha(x))=\varphi(x)1$
for any $x\in N$.

We define the right action $\alpha_0$ of $\mathcal{M}$ on $\hat M$ by $\alpha_0(x)=V(x\otimes1)V^*$. Remark that $\alpha_0$ satisfies
\begin{equation}\label{eq:action}
(\hat\delta\otimes\mathrm{id})\alpha_0=\mathrm{Ad}V_{13}(\mathrm{id}\otimes\alpha_0)\hat\delta
\end{equation}
since $V$ satisfies the pentagonal relation $V_{23}V_{12}=V_{12}V_{13}V_{23}$.

\begin{lemma}\label{lem:cond}
$E:=(\mathrm{id}\otimes h)\alpha_0$ is a normal conditional expectation from $\hat M$ onto $Z(\hat M):=\hat M\cap \hat M'$.
\end{lemma}
\begin{proof}
By definition, $E$ is normal. Moreover, for any $x\in \hat M$ we have
\begin{align*}
\alpha_0(E(x))
&=(\mathrm{id}\otimes\mathrm{id}\otimes h)((\alpha\otimes\mathrm{id})(\alpha_0(x)))\\
&=(\mathrm{id}\otimes\mathrm{id}\otimes h)((\mathrm{id}\otimes\delta)(\alpha_0(x)))\\
&=(\mathrm{id}\otimes\mathrm{id}\otimes h)(\alpha_0(x)\otimes1)\\
&=E(x)\otimes1.
\end{align*}
Thus, $E(x)\in\hat M^{\alpha_0}=Z(\hat M)$ and $E(E(x))=E(x)$. Namely, $E$ is a conditional expectation from $\hat M$ onto $Z(\hat M)$ by \cite[Theorem II.6.10.2]{Blackadar}. 
\end{proof}

The following is trivial.
\begin{lemma}\label{lem:cond_inv}
If $\varphi$ is $\alpha_0$-invariant, then $\varphi E=\varphi$.
\end{lemma}

For $\varphi\in\hat M_*$ define $T_\varphi\colon \hat M\to \hat M$ by $T_\varphi:=(\mathrm{id}\otimes\varphi)\hat\delta$. Then we have the following two lemmas.
\begin{lemma}\label{lem:center_preserving}
If $\varphi$ is $\alpha_0$-invariant, then $T_\varphi$ preserves the center $Z(\hat M):=\hat M\cap \hat M'$.
\end{lemma}
\begin{proof}
By the definition of $\alpha_0$, we have $Z(\hat M)=\hat M^{\alpha_0}$. Thus it suffices to show that $T_\varphi$ preserves $\hat M^{\alpha_0}$. By Equation \eqref{eq:action}, for any $x\in\hat M^{\alpha_0}$ we have
\begin{align*}
\alpha_0(T_\varphi(x))
&=\mathrm{Ad}V((\mathrm{id}\otimes \varphi\otimes\mathrm{id})(\hat\delta(x)\otimes1))\\
&=\mathrm{Ad}V((\mathrm{id}\otimes \varphi\otimes\mathrm{id})((\mathrm{id}\otimes\alpha_0)(\hat\delta(x))))\\
&=(\mathrm{id}\otimes \varphi\otimes\mathrm{id})(\mathrm{Ad}V((\mathrm{id}\otimes\alpha_0)(\hat\delta(x))))\\
&=(\mathrm{id}\otimes \varphi\otimes\mathrm{id})((\hat\delta\otimes\mathrm{id})(\alpha_0(x)))\\
&=T_\varphi(x)\otimes1.
\end{align*}
\end{proof}

\begin{remark}
There are many investigations of above transformations in probability theory and representation theory of compact (quantum) groups. See \cite{Biane91}, \cite{Izumi02}, \cite{INT}, \cite{Kuan18}, ... etc. 
\end{remark}

\begin{lemma}\label{lem:invariant_functional}
If $\varphi, \phi\in \hat M_*$ are $\alpha_0$-invariant, then $\varphi*\phi$ is also $\alpha_0$-invariant.
\end{lemma}
\begin{proof}
By Equation \eqref{eq:action},  for any $x\in \hat M$ we have
\begin{align*}
(\varphi*\phi\otimes\mathrm{id})(\alpha_0(x))
&=(\varphi\otimes\phi\otimes\mathrm{id})((\hat\delta\otimes\mathrm{id})(\alpha_0(x)))\\
&=(\varphi\otimes\phi\otimes\mathrm{id})(\mathrm{Ad}V((\mathrm{id}\otimes\alpha_0)(\hat\delta(x))))\\
&=(\varphi\otimes\mathrm{id})(\mathrm{Ad}V((\mathrm{id}\otimes\phi\otimes\mathrm{id})(\hat\delta(x)\otimes1)))\\
&=(\varphi\otimes\mathrm{id})(\alpha_0((\mathrm{id}\otimes\phi)(\hat\delta(x))))\\
&=(\varphi\otimes\phi)(\hat\delta(x))1\\
&=(\varphi*\phi)(x)1.
\end{align*}
\end{proof}

Since $h$ is a state, we have $1\in\mathfrak{n}_h$. Thus, we obtain $\hat\epsilon\in \hat M_*$ by $\hat\epsilon(x):=\langle x\eta_h(1), \eta_h(1)\rangle$ for any $x\in \hat M$. Then it is known that $\hat\epsilon$ becomes a counit of $\hat M$, i.e., $\hat\epsilon$ is multiplicative and satisfies that
$(\hat\epsilon\otimes\mathrm{id})\hat\delta=\mathrm{id}=(\mathrm{id}\otimes\hat\epsilon)\hat\delta$.
Remark that $\hat\epsilon(\hat\pi_h(\omega))=\omega(1)$ for any $\omega\in M_*$. 

\begin{lemma}\label{lem:counit_inv}
The counit $\hat\epsilon$ is $\alpha_0$-invariant.
\end{lemma}
\begin{proof}
By the definition of $V$, for any $\xi\in H_h$ we have $V(\eta_h(1)\otimes\xi)=\eta_h(1)\otimes\xi$. For any $\xi, \eta\in H_h$ we define $\omega_{\xi, \eta}\in M_*$ by $\omega_{\xi, \eta}(a):=\langle a\xi, \eta\rangle$ for any $a\in M$. Then, for any $x\in\hat M$
\begin{align*}
\omega_{\xi,\eta}((\hat\epsilon\otimes\mathrm{id})(\alpha_0(x)))
&=\langle V(x\otimes 1)V^*\eta_h(1)\otimes\xi, \eta_h(1)\otimes\eta\rangle\\
&=\langle x\eta_h(1)\otimes\xi, \eta_h(1)\otimes\eta\rangle\\
&=\hat\epsilon(x)\omega_{\xi, \eta}(1).
\end{align*}
Namely, we have $\omega((\hat\epsilon\otimes \mathrm{id})(\alpha_0(x)))=\hat\epsilon(x)\omega(1)$ for any $\omega\in M_*$, i.e., $(\hat\epsilon\otimes \mathrm{id})(\alpha_0(x))=\hat\epsilon(x)1$.
\end{proof}

For $\varphi\in \hat M_*$ we define
$\exp_*(\varphi):=\sum_{n=0}^\infty\varphi^{*n}/n!$,
where $\varphi^{*n}$ is the product $\varphi*\cdots*\varphi$ of $n$ copies of $\varphi$ and $\varphi^{*0}:=\widehat \epsilon$. We remark that the right-hand side is the limit of $\sum_{n=0}^N\varphi^{*n}/n!$ as $N\to\infty$ in the Banach space $\hat M_*$.

\begin{corollary}\label{cor:exp_inv}
If $\varphi\in \hat M_*$ is $\alpha_0$-invariant, then $\exp_*(\varphi)$ is also $\alpha_0$-invariant.
\end{corollary}
\begin{proof}
Since the set of $\alpha_0$-invariant functionals is closed in the norm topology of $\hat M_*$, it follows from Lemma \ref{lem:invariant_functional}
\end{proof}

\section{Center-preserving, right Haar weight invariant $CP_0$-semigroups}\label{sec:semigroups}
Let $M$ be a $W^*$-algebra and $(T_t)_{t\geq0}$ a semigroup of unital completely positive $\sigma$-weakly continuous linear maps $T_t\colon M\to M$. If $T_0=\mathrm{id}$ and $\lim_{t\searrow 0}\omega(T_t(x))=\omega(x)$ for any $x\in M$ and $\omega\in M_*$, then $(T_t)_{t\geq0}$ is called a \emph{$CP_0$-semigroup} on $M$. For a weight $\psi$ on $M$ the $CP_0$-semigroup $(T_t)_{t\geq0}$ is said to be \emph{$\psi$-invariant} if $\psi T_t=\psi$ for any $t\geq0$.

Let $\mathcal{M}=(M, \delta, R, \tau, h)$ be a Woronowicz algebra and $\hat{\mathcal{M}}=(\hat M, \hat\delta, \hat R, \hat\tau, \hat h)$ the dual Woronowicz algebra of $\mathcal{M}$. Assume that $h$ is a state. Thus, $\hat{\mathcal{M}}$ has a counit $\hat\epsilon$. A one-parameter family $(\varphi_t)_{t\geq0}$ of functionals in $\hat M_*$ is called a \emph{convolution semigroup} if $\varphi_0=\hat\epsilon$ and $\varphi_t*\varphi_s=\varphi_{t+s}$ for any $t, s\geq0$. If $\lim_{t\searrow0}\|\varphi_t-\hat\epsilon\|=0$, then $(\varphi_t)_{t\geq0}$ is said to be norm-continuous. If $\lim_{t\searrow0}\varphi_t(x)=\hat\epsilon(x)$ for any $x\in \hat M$, then $(\varphi_t)_{t\geq0}$ is said to be weakly continuous. 

\begin{proposition}
Let $(\varphi_t)_{t\geq0}$ be a convolution semigroup and $T_t:=(\mathrm{id}\otimes\varphi_t)\hat\delta$. Then the following are equivalent:
\begin{enumerate}
\item $(\varphi_t)_{t\geq0}$ is weakly continuous and $\varphi_t$ is a state for any $t\geq0$,
\item $(T_t)_{t\geq0}$ is a $CP_0$-semigroup on $\hat M$.
\end{enumerate}
\end{proposition}
\begin{proof}
It is clear that (1) implies (2). For any $t\geq0$ we have $\varphi_t=\hat\epsilon*\varphi_t=\hat\epsilon T_t$ since $(\hat\epsilon\otimes\mathrm{id})\hat\delta=\mathrm{id}$. Thus, (2) implies (1).
\end{proof}

A linear functional $\psi\in \hat M_*$ is said to be \emph{conditionally positive} if $\psi(x^*x)\geq0$ for any $x\in\mathrm{Ker}\hat\epsilon$. It is said to be hermitian if $\psi(x^*)=\overline{\varphi(x)}$ for any $x\in \hat M_*$.

\begin{theorem}\label{thm:preserving_semigroup}
Let $\psi\in \hat M_*$ be $\alpha_0$-invariant hermitian conditionally positive with $\psi(1)=0$, $\varphi_t:=\exp_*(t\psi)$, and $T_t:=(\mathrm{id}\otimes \varphi_t)\hat\delta$ for $t\geq0$. Then $(T_t)_{t\geq0}$ is a center-preserving, $\hat h\hat R$-invariant $CP_0$-semigroup.
\end{theorem}
\begin{proof}
By the Schoenberg correspondence (see \cite[Section 1.2]{Franz06}), $(\varphi_t)_{t\geq0}$ is a convolution semigroup of states. Moreover, $\varphi_t$ is $\alpha_0$-invariant for any $t\geq0$ by Corollary \ref{cor:exp_inv}. Thus, $T_t$ preserves the center $Z(\hat M)$ of $\hat M$ by Lemma \ref{lem:center_preserving}. Since $(\hat R\otimes\hat R)\hat\delta=\sigma_{\hat M}\hat\delta\hat R$, we have the right-invariance with respect to $\hat h\hat R$. Namely,
$
\hat h\hat R T_t
=(\hat h\hat R\otimes\varphi_t)\hat\delta
=\hat h \hat R
$ for any $t\geq0$. Moreover, $(T_t)_{t\geq0}$ is a $CP_0$-semigroup since $(\varphi_t)_{t\geq0}$ is weakly continuous.
\end{proof}

\begin{example}
Let $\varphi\in M_*$ be a $\alpha_0$-invariant state. For $c>0$ the linear functional $\psi:=c(\varphi-\hat\epsilon)$ is $\alpha_0$-invariant, hermitian, and conditionally positive with $\psi(1)=0$. Thus, by Theorem \ref{thm:preserving_semigroup}, $\psi$ generates a $CP_0$-semigroup. Then its generator $(\mathrm{id}\otimes \psi)\hat\delta$ is called a \emph{Poisson generator}.
\end{example}

By the example, for any $\alpha_0$-invariant state we can obtain a center-preserving, right Haar weight invariant $CP_0$-semigroup. As shown in the next section, the center of the group von Neumann algebra of the compact quantum group $G$ can be identified with the space $\ell^\infty(\widehat G)$ of bounded functions on the unitary dual $\widehat G$. Therefore, we obtain a Markov semigroup on $\widehat G$.

\section{Markov semigroups on unitary duals of compact quantum groups}\label{sec:cqg}
In this section, we apply the results in above sections to Woronowicz algebras given by compact quantum groups. Then we construct Markov semigroups on unitary duals of compact quantum groups. See Theorem \ref{thm:jun}. We first recall basic facts about compact quantum groups and give two Woronowicz algebras for a compact quantum group. 

Let $G=(C(G), \delta_G)$ be a compact quantum group, i.e., $C(G)$ is a unital $C^*$-algebra and $\delta_G\colon C(G)\to C(G)\otimes C(G)$ is a unital $*$-homomorphism satisfying that
\begin{itemize}
\item $(\delta_G\otimes\mathrm{id})\delta_G=(\mathrm{id}\otimes\delta_G)\delta_G$ as $*$-homomorphism from $C(G)$ to $C(G)\otimes C(G)\otimes C(G)$,
\item $(C(G)\otimes1)\delta_G(C(G)), (1\otimes C(G))\delta_G(C(G))\subset C(G)\otimes C(G)$ are dense,
\end{itemize}
where $\otimes$ denotes the operation of minimal tensor product of $C^*$-algebras. Then there exists a unique state $h_G$ on $C(G)$, called the \emph{Haar state} of $G$, such that 
\[(\omega\otimes h_G)\delta_G=(h_G\otimes\omega)\delta_G=\omega(1)h_G\] 
for any $\omega\in C(G)^*$. See \cite[Theorem 1.2.1]{NeshveyevTuset}. Let $(\pi_{h_G}, L^2(G), \eta_{h_G})$ be the unique GNS-triple associated with $h_G$. Then we denote by $L^\infty(G)$ the von Neumann algebra generated by $\pi_{h_G}(C(G))$. In this paper, we assume that $h_G$ is faithful and identify $C(G)$ with a subalgebra of $B(L^2(G))$. There exists a unitary operator $V_G$ on $L^2(G)\otimes L^2(G)$ such that 
\[V_G(\eta_{h_G}(a)\otimes\xi)=(\pi_{h_G}\otimes\pi_{h_G})(\delta_G(a))\xi_{h_G}\otimes\xi\] 
for any $\xi\in L^2(G)$ and $a\in C(G)$, where $\xi_{h_G}:=\eta_{h_G}(1)$. Then we obtain a comultiplication and Haar state of $L^\infty(G)$ by
\[\delta_G(a):=V_G(a\otimes 1)V_G^*,\quad h_G(a):=\langle a\xi_{h_G}, \xi_{h_G}\rangle\] 
for any $a\in L^\infty(G)$, respectively. We remark that $\delta_G$ and $h_G$ are the extensions of comultiplication and Haar state of $C(G)$, respectively. Thus, we use the same notations.

Let $H$ be a finite-dimensional Hilbert space and $U$ a \emph{unitary corepresentation} of $G$ on $H$, i.e., $U$ is a unitary element in $B(H)\otimes C(G)$ satisfying that $(\mathrm{id}\otimes \delta_G)(U)=U_{12}U_{13}$. For any linear functional $\varphi$ on $B(H)$ we call $(\varphi\otimes\mathrm{id})(U)$ a matrix coefficient of $U$. We denote by $A(G)$ the linear subspace of $C(G)$ generated by all matrix coefficients of finite-dimensional unitary corepresentations of $G$. Then $A(G)$ becomes $*$-subalgebra of $C(G)$. In this paper, we assume that $C(G)$ coincides with the universal $C^*$-algebra generated by $A(G)$. It is known that $A(G)$ has an $h_G$-invariant one-parameter automorphism group $\tau=(\tau_t)_{t\in\mathbb{R}}$ and a $h_G$-invariant unitary antipode $R\colon A(G)\to A(G)$. See \cite[Section 1.6, 1.7]{NeshveyevTuset}. Thus, there exist unitary operators $U_t$ for any $t\in \mathbb{R}$ and a conjugate linear operator $\hat J$ on $L^2(G)$ satisfying
\[U_t\eta_{h_G}(a)=\eta_{h_G}(\tau_t(a)),\quad \hat J\eta_{h_G}(a)=\eta_{h_G}(R(a))\]
for any $a\in A(G)$. Therefore, $L^\infty(G)$ has an one-parameter automorphism group $\tau^G=(\tau^G_t)_{t\in\mathbb{R}}$ and a unitary antipode $R_G\colon L^\infty(G)\to L^\infty(G)$ defined by
\[\tau^G_t(x):=U_txU_t^*,\quad R_G(x):=\hat Jx^*\hat J\]
for any $x\in L^\infty(G)$. Then $\mathcal{L}^\infty(G)=(L^\infty(G), \delta_G, R_G, \tau^G, h_G)$ becomes a Woronowicz algebra and $V_G$ is its multiplicative unitary. See \cite[Section 5]{MasudaNakagami} for more details. We denote by $\mathcal{W}^*(G)=(W^*(G), \hat\delta_G, \hat R_G, \hat\tau_G, \hat h_G)$ the dual Woronowicz algebra. Remark that $\mathcal{W}^*(G)$ has a counit, denoted by $\hat\epsilon_G$, since $h_G$ is bounded. 

Let $\widehat G$ be the set of all equivalence classes of irreducible unitary correpresentations of $G$. For any $\alpha\in\widehat G$ we fix a representative $(U_\alpha, H_\alpha)$ of $\alpha$. We denote by $\{f^G_z\}_{z\in\mathbb{C}}$ the \emph{Woronowicz characters}, which are distinguished multiplicative states on $A(G)$. See \cite[Section 1.7]{NeshveyevTuset}. We define $F_\alpha^z:=(\mathrm{id}\otimes f^G_z)(U_\alpha)$ for any $z\in \mathbb{C}$ and $d_q(\alpha):=\mathrm{Tr}_{H_\alpha}(F_\alpha)$. Then there exists a $*$-isomorphism
$\Phi_G\colon W^*(G)\cong \ell^\infty\mathchar`-\bigoplus_{\alpha\in \widehat G}B(H_\alpha)$.
Moreover, the deformation automorphism group $\hat\tau^G$ on $W^*(G)$ is given by 
$\hat\tau^G_t=\Phi_G^{-1}\circ(\prod_{\alpha\in\widehat G}\mathrm{Ad}F^{\mathrm{i}t})\circ\Phi_G$
for any $t\in\mathbb{R}$.

The following is our formulation of quantization of characters in \cite{Sato1}.
\begin{definition}
A normal $\hat\tau^G$-KMS (Kubo--Martin--Schwinger) state $\chi$ on $W^*(G)$ with inverse temperature -1 is called a \emph{quantized character} of $G$. Let $\mathrm{Ch}(G)$ be the set of all quantized characters of $G$.
\end{definition}

\begin{remark}
The $\mathrm{Ch}(G)$ is a convex set and its extreme points $\mathrm{ex}(\mathrm{Ch}(G))$ coincides with $\widehat G$ as sets. More precisely, for any $\alpha\in\widehat G$ the corresponding extreme quantized character $\chi_\alpha$ is given by $\chi_\alpha(x)=\mathrm{Tr}_{H_\alpha}(F_\alpha x_\alpha)/d_q(\alpha)$, where $x=\Phi_G^{-1}((x_\beta)_{\beta\in\widehat G})\in W^*(G)$. For the trivial corepresentation $U_{tr}=1\otimes1_{C(G)}$ we have the associated quantized character $\chi_{tr}$ and $\chi_{tr}=\hat\epsilon_G$. Moreover, every quantized character is a convex combination of extreme ones (see \cite[Lemma 2.2]{Sato1}). Thus, there is an affine bijection between $\mathrm{Ch}(G)$ and $\mathcal{M}_p(\widehat G)$.
\end{remark}

 
Recall that the right action $\alpha_0\colon W^*(G)\to W^*(G)\bar\otimes L^\infty(G)$ of $\mathcal{L}^\infty(G)$ on $W^*(G)$ is defined by $\alpha_0(x):=V_G(x\otimes 1)V_G^*$ for any $x\in W^*(G)$. Then we obtain the following lemma by \cite[Lemma 2.1.(4)]{Izumi02}.
 \begin{lemma}\label{lem:char_inv}
 Any quantized character $\chi\in \mathrm{Ch}(G)$ is $\alpha_0$-invariant.
 \end{lemma}

\begin{corollary}\label{cor:acid}
Let $\chi\in \mathrm{Ch}(G)$ and $\varphi_t:=\exp_*(t(\chi-\hat\epsilon_G))$ for any $t\geq0$. Then the $CP_0$-semigroup $(T_t)_{t\geq0}$ defined by $T_t:=(\mathrm{id}\otimes\varphi_t)\hat\delta_G$ is center-preserving and $\hat h_G\hat R_G$-invariant.
\end{corollary}
\begin{proof}
It follows from Lemma \ref{lem:char_inv}, Lemma \ref{lem:counit_inv} and Theorem \ref{thm:preserving_semigroup}.
\end{proof}

This corollary induces Markov dynamics on $\widehat G$. The following gives another algebraic interpretation of those dynamics on $\widehat G$. See also Proposition \ref{prop:dynamics_char}.

\begin{lemma}\label{lem:yosa}
Let $\chi\in \mathrm{Ch}(G)$ and $\varphi_t:=\exp_*(t(\chi-\hat\epsilon_G))$ for any $t\geq0$. Then $\varphi_t\in\mathrm{Ch}(G)$.
\end{lemma}
\begin{proof}
We fix two elements $x, y\in W^*(G)$. Since $\chi^{*k}$ is a $\hat\tau^G$-KMS state for any $k\geq0$ (see \cite[Lemma 3.2]{Sato3}), there exists a bonded continuous function $F_k$ on $\bar{D}=\{z\in\mathbb{C}\mid -1\leq\mathrm{Im}(z)\leq 0\}$, which is analytic in $D=\{z\in\mathbb{C}\mid -1<\mathrm{Im}(z)<0\}$ such that for any $s\in\mathbb{R}$
\[F_k(s)=\chi^{*k}(x\hat\tau^G_s(y)),\quad F_k(s-\mathrm{i})=\chi^{*k}(\hat\tau^G_s(y)x).\]
By \cite[Proposition 5.3.7]{BratteliRobinson2}, we have $|F_k(z)|<\|x\|\|y\|$ on  $\bar{D}$. Thus,
\[F:=\sum_{n=0}^\infty\frac{t^n}{n!}\sum_{k=0}^n(-1)^k\binom{n}{k}F_k.\]
converges uniformly in $\bar{D}$. Then $F$ satisfies that $|F|\leq e^{2t}\|x\|\|y\|$ on $\bar{D}$ and analytic in $D$. Moreover, since $(\chi-\hat\epsilon_G)^{*n}=\sum_{k=0}^n(-1)^k\binom{n}{k}\chi^{*k}$, we have
\[F(s)=\varphi_t(x\hat\tau^G_s(y)),\quad F(s-\mathrm{i})=\varphi_t(\hat\tau^G_s(y)x).\]
Namely, $\varphi_t$ is a normal $\hat\tau^G$-KMS state, i.e., $\varphi_t\in\mathrm{Ch}(G)$.
\end{proof}

We define $\pi_G\colon L^\infty(G)_*\to W^*(G)$ by $\pi_G(\omega)=(\mathrm{id}\otimes\omega)(V_G)$ for any $\omega\in L^\infty(G)_*$. Then for any $\varphi\in W^*(G)_*$ we define $\pi_{G*}(\varphi)\in L^\infty(G)$ by $\omega(\pi_{G*}(\varphi))=\varphi(\pi_G(\omega))$ for any $\omega\in L^\infty(G)$. Namely, $\pi_{G*}(\varphi)=(\varphi\otimes\mathrm{id})(V_G)$. Since $\pi_{G*}\colon W^*(G)_*\to L^\infty(G)$ is a homomorphism, we have the following lemma.
\begin{lemma}\label{lem:tutumi}
For any $\psi\in W^*(G)_*$ we have $\pi_{G*}({\exp_*(\psi)})=e^{\pi_{G*}(\psi)}=\sum_{n=0}^\infty\pi_{G*}(\psi)^n/n!$. 
\end{lemma}

Let $T_\varphi:=(\mathrm{id}\otimes\varphi)\hat\delta_G$ for any $\varphi\in W^*(G)_*$. For any $\omega\in L^\infty(G)_*$ and $a\in L^\infty(G)$ we define $a\omega, \omega a\in L^\infty(G)_*$ by $[a\omega](b):=\omega(ba)$ and $[\omega a](b):=\omega(ab)$ for any $b\in L^\infty(G)$.
\begin{lemma}\label{lem:meg}
For any $\varphi\in W^*(G)_*$ and $\omega\in L^\infty(G)_*$ we have $T_\varphi(\pi_G(\omega))=\pi_G(\omega\pi_{G*}(\varphi))$.
\end{lemma}
\begin{proof}
By the pentagonal relation of $V_G$, we have $(\hat\delta_G\otimes\mathrm{id})(V_G)=V_{G13}V_{G23}$. Thus,
\[T_\varphi(\pi_G(\omega))=(\mathrm{id}\otimes\varphi\otimes\omega)(V_{G13}V_{G23})=(\mathrm{id}\otimes\omega)(1\otimes\pi_{G*}(\varphi)V_G)=\pi_G(\omega\pi_{G*}(\varphi)).\]
\end{proof}

By Corollary \ref{cor:acid}, for any $\chi\in\mathrm{Ch}(G)$ we obtain the center-preserving $CP_0$-semigroup on $W^*(G)$. Now we study the restriction of such $CP_0$-semigroups to the center $Z(W^*(G))$. For any $\xi\in L^2(G)$ and $\alpha\in\widehat G$ we define a linear operator $\hat\xi(\alpha)$ on $H_\alpha$ by $\hat\xi(\alpha):=(\mathrm{id}\otimes\omega_\xi)(U_\alpha)^*$, where $\omega_\xi\in C(G)^*$ is defined by $\omega_\xi(a):=\langle \eta_G(a), \xi\rangle$ for any $a\in C(G)$. Then the Peter--Weyl type theorem (see Theorem \ref{thm:PW}) states that for any $\xi, \eta\in L^2(G)$
\[\langle\xi, \eta\rangle=\sum_{\alpha\in\widehat G}d_q(\alpha)\mathrm{Tr}_{H_\alpha}(F_\alpha\hat\eta(\alpha)^*\hat\xi(\alpha)).\]
Moreover, there exists a $*$-isomorphism $\Psi_G\colon \ell^\infty(\widehat G)\to Z(W^*(G))$ such that
\[\langle\Psi_G(\kappa)\xi,\eta\rangle=\sum_{\alpha\in\widehat G}d_q(\alpha)\kappa(\alpha)\mathrm{Tr}_{H_\alpha}(F_\alpha\hat\eta(\alpha)^*\hat\xi(\alpha))\]
for any $\xi, \eta\in L^2(G)$ and $\kappa\in\ell^\infty(\widehat G)$.

By Lemma \ref{lem:center_preserving}, the linear map $T_\varphi:=(\mathrm{id}\otimes\varphi)\hat\delta_G$ preserves $Z(W^*(G))$ if $\varphi$ is a $\alpha_0$-invariant state. Thus, we have a Markov operator $Q_\varphi:=\Psi_G^{-1}T_\varphi|_{Z(W^*(G))}\Psi_G$ on $\ell^\infty(\widehat G)$. Then we define the mapping $P\in\mathcal{M}_p(\widehat G)\to PQ_\varphi\in\mathcal{M}_p(\widehat G)$ by $(PQ_\varphi, \kappa)=(P, Q_\varphi\kappa)$ for any $\kappa\in\ell^\infty(\widehat G)$, where $(\,\cdot\,,\,\cdot\,)$ is the natural pairing of $\mathcal{M}_p(\widehat G)$ and $\ell^\infty(\widehat G)$. Recall that there exists an affine homeomorphism from $\mathcal{M}_p(\widehat G)$ to $\mathrm{Ch}(G)$ by irreducible decomposition. More precisely, for any $P\in\mathcal{M}_p(\widehat G)$ the corresponding quantized character $\chi_P\in\mathrm{Ch(G)}$ is given as $\chi_P=\sum_{\alpha\in\widehat G}P(\alpha)\chi_\alpha$. The following describes the dynamics $Q_\varphi$ on $\widehat G$ in terms of quantized characters.

\begin{proposition}\label{prop:dynamics_char}
Let $\varphi\in W^*(G)_*$ be a $\alpha_0$-invariant state and $P\in\mathcal{M}_p(\widehat G)$. Then $\chi_{PQ_\varphi}=\chi_P*\varphi$.
\end{proposition}
\begin{proof}
Since $\chi_{PQ_\varphi}, \chi_P*\varphi$ are $\alpha_0$-invariant, by Lemma \ref{lem:cond_inv}, it suffices to show that $\chi_{PQ_\varphi}=\chi_P*\varphi$ on $Z(W^*(G))$. For any $\alpha\in \widehat G$ we have
$\chi_P*\varphi(\Psi_G(\delta_\alpha))=\chi_P(\Psi_G(Q_\varphi\delta_\alpha))=\chi_{PQ_\varphi}(\Psi_G(\delta_\alpha))$. Since $\chi_{PQ_\varphi}, \chi_P*\varphi$ are normal, they are equal each other on $Z(W^*(G))$. Thus, by Lemma \ref{lem:cond_inv}, $\chi_{PQ_\varphi}=\chi_P*\varphi$ on $W^*(G)$.
\end{proof}

\begin{remark}
Let $tr\in\widehat G$ be the equivalence class of the trivial corepresentation of $G$. Then, by the above proposition, we have $\varphi=\chi_{\delta_{tr}Q}$ for any $\alpha_0$-invariant state. In particular, any $\alpha_0$-invariant states are quantized character of $G$. This is a natural characterization of quantized characters since the action $\alpha_0$ is an analog of the adjoint action of group on itself (see \cite{ES}).
\end{remark}

\begin{remark}\label{rem:q-aina}
For any $\alpha\in\widehat G$ we define $X_\alpha:=\pi_{G*}(\chi_\alpha)$. By Corollary \ref{cor:hinata}, we have $X_\alpha=(\chi_\alpha^0\otimes \mathrm{id})(U_\alpha)$, where $\chi_\alpha^0\in B(H_\alpha)^*$ is given as $\chi_\alpha^0(A):=\mathrm{Tr}_{H_\alpha}(F_\alpha A)/d_q(\alpha)$. Then, by \cite[Theorem 1.4.2]{NeshveyevTuset}, for any $\alpha, \beta\in\widehat G$ we can show that 
\[[h_GX_\alpha^*](X_\beta)=\frac{\delta_{\alpha, \beta}}{d_q(\beta)^2},\quad \pi_\beta(h_GX_\alpha^*)=\delta_{\alpha, \beta}\frac{1}{d_q(\beta)^2}1_{H_\beta},\]
where $\pi_\beta(\omega):=(\mathrm{id}\otimes\omega)(U_\beta)$ for any $\omega\in L^\infty(G)_*$. Therefore, by the identification of $W^*(G)$ and $\ell^\infty\mathchar`-\bigoplus_{\alpha\in\widehat G}B(H_\alpha)$ (see Appendix \ref{app:FA}), if $\omega=\sum_{\alpha\in\widehat G}d_q(\alpha)^2\kappa(\alpha)h_GX_\alpha^*$ in $L^\infty(G)_*$, then $\kappa\in \ell^\infty(\widehat G)$ and
\begin{align*}
\langle\Psi_G(\kappa)\xi, \eta\rangle
&=\sum_{\alpha\in \widehat G}d_q(\alpha)\kappa(\alpha)\mathrm{Tr}_{H_\alpha}(F_\alpha\hat\eta(\alpha)^*\hat\xi(\alpha))\\
&=\sum_{\alpha\in \widehat G}d_q(\alpha)\mathrm{Tr}_{H_\alpha}(F_\alpha\hat\eta(\alpha)^*\widehat{\pi_G(\omega)\xi}(\alpha))\\
&=\langle\pi_G(\omega)\xi, \eta\rangle.
\end{align*}
Namely, we have $\Psi_G(\kappa)=\pi_G(\omega)$.
\end{remark}

\begin{lemma}\label{lem:4bro}
Let $\varphi\in W^*(G)_*$ be $\alpha$-invariant. Then $Q_\varphi(\delta_\beta)(\alpha)=d_q(\beta)^2h_G(X_\beta^*X_\alpha\pi_{G*}(\varphi))$ for any $\alpha, \beta\in\widehat G$. 
\end{lemma}
\begin{proof}
By Proposition \ref{prop:dynamics_char} and Remark \ref{rem:q-aina}, 
\[Q_\varphi(\delta_\beta)(\alpha)=(\chi_\alpha*\varphi)(\Psi_G(\delta_\beta))=d_q(\beta)^2(\chi_\alpha*\varphi)(\pi_G(h_GX_\beta^*))=d_q(\beta)^2h_G(X_\beta^*X_\alpha\pi_{G*}(\varphi)).\]
\end{proof}

By the above lemma, the action $Q_\varphi$ on $\ell^\infty(\widehat G)$ coincides with the action of the Markov kernel $\mathbb{Q}_\varphi$ defined by $\mathbb{Q}_\varphi(\alpha, \beta)=d_q(\beta)^2h_G(X_\beta^*X_\alpha\pi_{G*}(\varphi))$.

\begin{lemma}\label{lem:maru}
For any $\omega\in L^\infty(G)_*$ it holds true that $\lim_{\alpha\in\widehat G}\|\pi_\alpha(\omega)\|=0$. In particular, $\kappa\in c_0(\widehat G)$ if $\Psi_G(\kappa)=\pi_G(\omega)$ for some $\omega\in L^\infty(G)_*$.
\end{lemma}
\begin{proof}
By \cite[Theorem III 2.1.4]{Blackadar}, there exist sequences $(\xi_n)_{n=1}^\infty, (\eta_n)_{n=1}^\infty$ in $L^2(G)$ such that $\sum_{n=1}^\infty\|\xi_n\|_2^2<\infty$, $\sum_{n=1}^\infty\|\eta_n\|_2^2<\infty$, and $\omega(x)=\sum_{n=1}^\infty\langle x\xi_n, \eta_n\rangle$ for any $x\in L^\infty(G)$. For any $\epsilon>0$ we fix sufficiently large $N$ such that
\[\left(\sum_{n\geq N+1}\|\xi_n\|_2^2\right)^{1/2}\left(\sum_{n\geq N+1}\|\eta_n\|_2^2\right)^{1/2}<\epsilon/2.\]
Since $\{\eta_G(a)\mid a\in A(G)\}$ is dense in $L^2(G)$, there exist $a_n, b_n\in A(G)$ such that
\[\|\xi_n-\eta_G(a_n)\|_2<\frac{\epsilon}{4N\|\eta_n\|_2},\quad \|\eta_n-\eta_G(b_n)\|_2<\frac{\epsilon}{4N\|\eta_G(a_n)\|_2}.\]
Then $\omega_\epsilon:=\sum_{n=1}^Nb_n^*h_Ga_n$ satisfies that $\|\omega-\omega_\epsilon\|_1<\epsilon$. Moreover, we have 
\[\|\Phi_G(\pi_G(\omega))-\Phi_G(\pi_G(\omega_\epsilon))\|\leq\|\omega-\omega_\epsilon\|_1<\epsilon.\]
Since $a_n, b_n\in A(G)$, we have
\[\Phi_G(\pi_G(\omega_\epsilon))\in\bigoplus_{\alpha\in\widehat G}B(H_\alpha)\subset c_0\mathchar`-\bigoplus_{\alpha\in\widehat G}B(H_\alpha).\]
Therefore, we have $\lim_{\alpha\in\widehat G}\|\pi_\alpha(\omega)\|=0$ since $c_0\mathchar`-\bigoplus_{\alpha\in\widehat G}B(H_\alpha)$ is norm-closed.

The rest of the statement immediately follows from Remark \ref{rem:q-aina}.
\end{proof}

Here we show the first main theorem.
\begin{theorem}\label{thm:jun}
Let $\chi\in\mathrm{Ch}(G)$, $\chi_t:=\exp_*(t(\chi-\hat\epsilon_G))$, and $T^\chi_t:=(\mathrm{id}\otimes\chi_t)\hat\delta_G$ for any $t\geq0$. Then the Markov kernels $\mathbb{Q}^\chi_t$ associated with $Q^\chi_t:=\Psi_G^{-1}T^\chi_t|_{Z(W^*(G))}\Psi_G$ form a Markov semigroup on $\widehat G$. Moreover, $(\mathbb{Q}^\chi_t)_{t\geq0}$ is a Feller semigroup and its generator $\mathbb{L}_\chi$ is given by 
\[\mathbb{L}_\chi(\alpha, \beta):=d_q(\beta)^2h_G(X_\beta^*X_\alpha(\pi_{G*}(\chi)-1))\] 
for any $\alpha, \beta\in \widehat G$. In particular, the generator $\mathbb{L}_\chi$ is bounded.
\end{theorem}
\begin{proof}
First we show that $(\mathbb{Q}^\chi_t)_{t\geq0}$ is a Feller semigroup, i.e., $(\mathbb{Q}^\chi_t)_{t\geq0}$ preserves $c_0(\widehat G)$. Let $\mathcal{A}:=\left\{\kappa\in c_0(\widehat G)\,\middle|\, \sum_{\alpha\in \widehat G}d_q(\alpha)|\kappa(\alpha)|<\infty\right\}$. Since $c_0(\widehat G)$ is norm-closed and $\mathcal{A}$ is dense in $c_0(\widehat G)$ by the Stone--Weierstrass theorem, it suffices to show that $\mathbb{Q}^\chi_t\kappa\in c_0(\widehat G)$ for any $t\geq0$ and $\kappa\in \mathcal{A}$. Since $\|d_q(\alpha)h_GX_\alpha^*\|_1\leq\|\eta_G(d_q(\alpha)X_\alpha)\|=1$, for $\kappa\in \mathcal{A}$ the functional $\omega_\kappa:=\sum_{\alpha\in\widehat{G}}d(\alpha)^2\kappa(\alpha)h_GX_\alpha^*$ converges in $L^1(G)_*$. Then, by Remark \ref{rem:q-aina} and Lemma \ref{lem:meg}, \ref{lem:maru}, we have $\mathbb{Q}^\chi_t\kappa=\Psi_G^{-1}\pi_G(\omega_\kappa\pi_{G*}(\chi_t))\in c_0(\widehat G)$.

By Lemma \ref{lem:tutumi}, \ref{lem:4bro}, for any $\alpha, \beta\in \widehat G$ we have
\begin{align*}
\frac{1}{t}(\mathbb{Q}^\chi_t\delta_\beta-\delta_\beta)(\alpha)
&=d_q(\beta)^2h_G(X_\beta^*X_\alpha\frac{e^{t(\pi_G(\chi)-1)}}{t})\\
&\to d_q(\beta)^2h_G(X_\beta^*X_\alpha(\pi_G(\chi)-1)) \quad(t\searrow0).
\end{align*}
Therefore, the generator $\mathbb{L}_\chi$ of $(\mathbb{Q}^\chi_t)_{t>0}$ is given as $\mathbb{L}_\chi(\alpha, \beta):=d(\beta)^2h_G(X_\beta^*X_\alpha(\pi_{G*}(\chi)-1))$ for any $\alpha, \beta\in \widehat G$. In particular, $\mathbb{L}_\chi=\mathbb{Q}_\chi-\mathrm{id}$, that is, $\mathbb{L}_\chi$ is bounded,
\end{proof}

By \cite[Lemma 3.2]{Sato3}, $\chi_\alpha*\chi_\gamma$ is also quantized character of $G$ for any $\alpha, \gamma\in\widehat G$. We define $N^\beta_{\alpha, \gamma}$ by $(d_q(\alpha)\chi_\alpha)*(d_q(\gamma)\chi_\gamma)=\sum_{\beta\in\widehat{G}}N^\beta_{\alpha,\beta}d_q(\beta)\chi_\beta$ be its irreducible decomposition.

\begin{corollary}\label{cor:pubg}
Assume that $\chi\in\mathrm{Ch}(G)$ has an irreducible decomposition $\chi=\sum_{\alpha\in\widehat G}P_\chi(\alpha)\chi^\alpha$. Then the generator $\mathbb{L}_\chi$ of the Feller semigroup in Theorem \ref{thm:jun} is given by
\[\mathbb{L}_\chi(\alpha, \beta)=\frac{d_q(\beta)}{d_q(\alpha)}\sum_{\gamma\in\widehat G}\frac{P_\chi(\gamma)N_{\alpha, \gamma}^\beta}{d_q(\gamma)}-\delta_{\alpha,\beta}.\]
\end{corollary}
\begin{proof}
By Theorem \ref{thm:jun} and Propoisition \ref{prop:dynamics_char}, we have
\[\mathbb{L}_\chi(\alpha, \beta)=\mathbb{Q}_\chi(\alpha, \beta)-\delta_{\alpha, \beta}=\chi_\alpha*\chi(\Psi_G(\delta_\beta))-\delta_{\alpha,\beta}=\frac{d_q(\beta)}{d_q(\alpha)}\sum_{\gamma\in\widehat G}\frac{P_\chi(\gamma)N_{\alpha, \gamma}^\beta}{d_q(\gamma)}-\delta_{\alpha,\beta}.\]
\end{proof}

If $\widehat G$ is separable, we have the following by \cite[Theorem 4.2.7]{EK}.
\begin{corollary}
For any probability measure $P$ on $\widehat G$ and $\chi\in\mathrm{Ch}(G)$ there exists a Markov process $(X_t)_{t\geq0}$ corresponding to $(\mathbb{Q}^\chi_t)_{t\geq0}$ with initial distribution $P$ and c\`adl\`ag sample paths.
\end{corollary}

\section{Examples from unitary groups $U(N)$}\label{sec:unitary}
In this section, we apply the above results to the case of unitary groups $U(N)$ and certain characters. When we deal with an ordinary topological group $G$, quantized characters can be regarded as functions on $G$. A continuous function $\chi$ on $G$ is called a \emph{character} of $G$ if $\chi$ is positive-definite, $\chi(gh)=\chi(hg)$ for any $g, h\in G$, and $\chi(e)=1$. We equip the set of characters with the topology of uniform convergence on compact subsets. Then the set of characters is homeomorphic to the set of quantized characters (i.e., the set of normal tracial states on $W^*(U(\infty))$). See \cite[Remark 3.2]{Sato3}.

Here we recall some results of the representation theory of $U(N)$. It is well known that $\widehat{U(N)}$ can be identified with $\mathbb{S}_N:=\{\lambda=(\lambda_1\geq\cdots\geq\lambda_N)\in\mathbb{Z}^N\}$. Then for any $\lambda\in\mathbb{S}_N$ the corresponding irreducible character $\chi_\lambda$ of $U(N)$ is given by the Schur polynomial $s_\lambda$, i.e.,
\[\chi_\lambda(U)=\frac{s_\lambda(z_1,\dots, z_N)}{s_\lambda(1,\dots, 1)},\]
where $z_1,\dots, z_N$ are eigenvalues of $U\in U(N)$. See e.g. \cite{Zelobenko}.

We define the \emph{infinite-dimensional unitary group} $U(\infty)$ as the inductive limit $\varinjlim_NU(N)$ by natural embeddings
\[U\in U(N)\hookrightarrow \begin{bmatrix}U&0\\0&1\end{bmatrix}\in U(N+1).\]
Then the theorem due to Voiculescu \cite{Voiculescu76}, Vershik--Kerov \cite{VK82}, and Boyer \cite{Boyer83} states that any extreme characters of $U(\infty)$ are parametrized by the set $\Omega\subset \mathbb{R}_{\geq0}^\infty\times\mathbb{R}_{\geq0}^\infty\times\mathbb{R}_{\geq0}^\infty\times\mathbb{R}_{\geq0}^\infty\times\mathbb{R}_{\geq0}\times\mathbb{R}_{\geq0}$ consisting of $\omega=(\alpha^+, \beta^+, \alpha^-, \beta^-, \delta^+, \delta^-)$ such that
\[\alpha^\pm=(\alpha^\pm_1\geq\alpha^\pm_2\geq\cdots),\quad \beta^\pm=(\beta^\pm_1\geq\beta^\pm_2\geq\cdots),\]
\[\sum_{i=1}^\infty(\alpha^\pm_i+\beta^\pm_i)<\infty,\quad \beta^+_1+\beta^-_1\leq1,\quad \gamma^\pm:=\delta^\pm-\sum_{i=1}^\infty(\alpha^\pm_i+\beta^\pm_i)\geq0.\]
More precisely, for any $\omega\in\Omega$ the corresponding extreme character $\chi_\omega$ is given by
\begin{equation}\label{eq:kit}
\chi_\omega(U)=\prod_{z}\Phi_\omega(z),
\end{equation}
where $z$ runs over all any eigenvalues of $U\in U(\infty)$, which are equal to 1 except for finitely many eigenvalues, and 
\[\Phi_\omega(z):=e^{\gamma^+(z-1)+\gamma^-(z^{-1}-1)}\prod_{i=1}^\infty\frac{1+\beta^+_i(z-1)}{1-\alpha^+_i(z-1)}\frac{1+\beta^-_i(z^{-1}-1)}{1-\alpha^-_i(z^{-1}-1)}.\]
Then we define $\varphi_\omega(n)$ by $\Phi_\omega(z)=\sum_{n\in\mathbb{Z}}\varphi_\omega(n)z^n$, where the expansion converges in a neighborhood of $\mathbb{T}:=\{z\in\mathbb{C}\mid|z|=1\}$. It is well known that $(\varphi_\omega(n))_{n\in\mathbb{Z}}$ is totally positive, that is, for any $N\geq1$ and $m_1>\cdots>m_N$ the determinant $\det[\varphi_\omega(m_i+j)]_{i,j=1}^N$ is nonnegative.

\begin{theorem}\label{thm:kuma}
For any $\omega\in\Omega$ let $\chi_N$ be the restriction of $\chi_\omega$ to $U(N)$. Then the generator $\mathbb{L}_{\chi_N}$ of Markov semigroup associated with $\chi_N$ in Theorem \ref{thm:jun} is given by 
\[\mathbb{L}_{\chi_N}(\lambda, \mu)=\frac{s_\mu(1,\dots, 1)}{s_\lambda(1,\dots, 1)}\det[\varphi_\omega(\mu_j-j-\lambda_i-i)]_{i,j=1}^N-\delta_{\lambda, \mu}\]
for any $\lambda, \mu\in \mathbb{S}_N$, where $s_\lambda(1,\dots, 1)$ is the specialization of the Schur polynomial $s_\lambda$ with $N$ variables. 
\end{theorem}
\begin{proof}
Remark that the dimension $d(\lambda)$ of irreducible representation associated with $\lambda\in\mathbb{S}_N$ coincides with $s_\lambda(1,\dots,1)$. Thus, by Corollary \ref{cor:pubg},
\begin{align*}
\mathbb{L}_{\chi_N}(\lambda, \mu)
&=\frac{s_\mu(1,\dots,1)}{s_\lambda(1,\dots, 1)}\frac{1}{N!}\int_{\mathbb{T}^N}\chi_N(\mathrm{diag}(z_1,\dots,z_N))s_\lambda(z)\overline{s_\mu(z)}|V_N(z)|^2\prod_{i=1}^N\frac{dz_i}{2\pi\mathrm{i}z_i}-\delta_{\lambda, \mu},
\end{align*}
where $\mathrm{diag}(z_1,\dots, z_N)$ is the diagonal matrix whose entries are $z_1,\dots, z_N$, $s_\lambda(z)=s_\lambda(z_1,\dots, z_N)$ and $V_N(z)=V_N(z_1,\dots, z_N)$. By Equation \eqref{eq:kit}, we have
\begin{align*}
\mathbb{L}_{\chi_N}(\lambda, \mu)
&=\frac{s_\mu(1,\dots,1)}{s_\lambda(1,\dots, 1)}\frac{1}{N!}\int_{\mathbb{T}^N}\left(\prod_{i=1}^N\Phi_\omega(z_i)\right)\det[z_i^{\lambda_j+N-j}]_{i,j=1}^N\det[z_i^{-\mu_j-N+j}]_{i,j=1}^N\prod_{i=1}^N\frac{dz_i}{2\pi\mathrm{i}z_i}-\delta_{\lambda, \mu}\\
&=\frac{s_\mu(1,\dots,1)}{s_\lambda(1,\dots, 1)}\frac{1}{N!}\sum_{\sigma, \tau\in S(N)}\mathrm{sgn}(\sigma\tau)\prod_{i=1}^N\int_{\mathbb{T}}\Phi_\omega(z)z^{-(\mu_{\tau(i)}-\tau(i)-\lambda_{\sigma(i)}+\sigma(i))}\frac{dz}{2\pi\mathrm{i}z}-\delta_{\lambda, \mu}\\
&=\frac{s_\mu(1,\dots, 1)}{s_\lambda(1,\dots, 1)}\det[\varphi_\omega(\mu_j-j-\lambda_i-i)]_{i,j=1}^N-\delta_{\lambda, \mu}
\end{align*}
\end{proof}

\section{Examples from quantum unitary groups $U_q(N)$}\label{sec:quantum_unitary}
In this section, we study the case of quantum unitary groups $U_q(N)$ and certain quantized characters. See \cite{NoumiYamadaMimachi}, \cite{KliSch} for the detailed definition and the representation theory of $U_q(N)$. First we recall some results of the representation theory of $U_q(N)$. It is known that $\widehat{U_q(N)}$ can be identified with $\mathbb{S}_N$. For $\lambda\in\mathbb{S}_N$ we denote by $\chi_{q, \lambda}$ the corresponding irreducible quantized character. Then the``restriction" of $\chi_{q, \lambda}$ to $\mathbb{T}^N$ is given as 
\[\frac{s_\lambda(q^{N-1}z_1, q^{N-3}z_2\dots,q^{-N+1}z_N)}{s_\lambda(q^{N-1}, q^{N-3},\dots, q^{-N+1})}.\]
See \cite{NoumiYamadaMimachi} and \cite{Sato1} for more details.

For any probability measure $P_N$ on $\mathbb{S}_N$ we define its $q^2$-Schur generating function $\mathcal{S}(z_1,\dots, z_N; P_N)$ by
\[\mathcal{S}(z_1,\dots, z_N;P_N):=\sum_{\lambda\in\mathbb{S}_N}P_N(\lambda)\frac{s_\lambda(z_1, q^{-2}z_2,\dots, q^{-2(N-1)}z_N)}{s_\lambda(1, q^{-2},\dots, q^{-2(N-1)})},\]
where the right-hand side absolutely converges on $\mathbb{T}^N$ (see \cite[Proposition 4.5]{Gorin12}). Remark that $q^2$-Schur generating functions uniquely determine probability measures on $\mathbb{S}_N$. We denote by $U_q(\infty)$ the infinite-dimensional quantum unitary group, which is the inductive limit of $U_q(N)$ in the sense of \cite{Sato3}. Then we denote by $\Theta^\infty_N\colon W^*(U_q(N))\to W^*(U_q(\infty))$ the embeddings of group von Neumann algebras (see Section \ref{sec:ind}). By \cite[Proposition 4.6]{Gorin12}, \cite[Proposition 2.6]{Sato1}, there exists a one-to-one correspondence between $\chi\in\mathrm{Ch}(U_q(\infty))$ and the set of all sequences $(P_N)_{N\geq1}$ of probability measures $P_N$ on $\mathbb{S}_N$ satisfying that 
\begin{equation}\label{eq:coh}
\mathcal{S}(z_,\dots, z_N, 1; P_{N+1})=\mathcal{S}(z_1,\dots, z_N; P_N)
\end{equation} 
for any $N\geq1$. More precisely, the corresponding probability measures $P_N$ are given by the irreducible decompositions $\chi\Theta^\infty_N=\sum_{\lambda\in\mathbb{S}_N}P_N(\lambda)\chi_{q, \lambda}$. The following lemma gives examples of $q^2$-Schur generating functions satisfying Equation \eqref{eq:coh}.

\begin{lemma}\label{lem:q-char}
For $N\geq1$ let $\omega=(\alpha^+,\beta^+,\alpha^-,\beta^-,\gamma^+,\gamma^-)\in\Omega$ such that $\Phi_\omega(q^{-2(i-1)})>0$ for $i=1,\dots,N$ and the Laurent expansion of $\Phi_\omega(z)$ converges in annulus containing $1,\dots, q^{-2(N-1)}$. Then for $k=1,\dots, N$ 
\begin{equation}\label{eq:q-char}
\mathcal{F}(z_1,\dots, z_k):=\prod_{i=1}^k\frac{\Phi_\omega(q^{-2(i-1)}z_i)}{\Phi_\omega(q^{-2(i-1)})}
\end{equation}
is a $q^2$-Schur generating function of probability measure $P_k$ on $\mathbb{S}_k$ given by
\[P_k(\lambda):=\frac{\det[\varphi_\omega(\lambda_i-i+j)]_{i,j=1}^k}{\prod_{i=1}^k\Phi_\omega(q^{-2(i-1)})}s_\lambda(1, q^{-2},\dots, q^{-2(k-1)})\]
for any $\lambda\in\mathbb{S}_k$. Moreover, Equation \eqref{eq:coh} holds true for $k=1,\dots, N$.
\end{lemma}

\begin{remark}
For instance we have $\Phi_\omega(q^{-2(i-1)})>0$ if $\alpha^\pm_i$ and $\beta^\pm_i$ are zero except for finitely many $i$ and $\beta^-_1<(1-q^{2(N-1)})^{-1}, \alpha^+_1<(q^{-2(N-1)}-1)^{-1}$. Moreover, if $\alpha^-_i$ and $\beta^+_i$ are zero except fo finitely many $i$, $\alpha^+=\beta^-=(0,0,\dots)$, then $\mathcal{F}(z_1,\dots, z_N)$ is a $q^2$-Schur generating function satisfying Equation \eqref{eq:coh} for any $N\geq1$. Thus, there exists a corresponding quantized character $\chi_{q,\omega}$ of $U_q(\infty)$.
\end{remark}

\begin{proof}
Since Equation \eqref{eq:coh} is clear, it suffices to show that $\mathcal{F}(z_1,\dots, z_k)$ is a $q^2$-Schur generating function of some probability measure on $\mathbb{S}_k$. First we have
\begin{align*}
&V(z_1,q^{-2}z_2,\dots, q^{-2(k-1)}z_k)\mathcal{F}(z_1,\dots,z_k)\\
&=\frac{1}{\prod_{i=1}^k\Phi_\omega(q^{-2(i-1)})}\sum_{\sigma\in S(k)}\sum_{k_1,\dots,k_k\in\mathbb{Z}}\mathrm{sgn}(\sigma)\prod_{i=1}^k\varphi_\omega(k_i)(q^{-2(i-1)}z_i)^{k-\sigma(i)+k_i}\\
&=\frac{1}{\prod_{i=1}^k\Phi_\omega(q^{-2(i-1)})}\sum_{\sigma\in S(k)}\sum_{m_1,\dots,m_k\in\mathbb{Z}}\mathrm{sgn}(\sigma)\prod_{i=1}^k\varphi_\omega(m_i+\sigma(i))(q^{-2(i-1)}z_i)^{k+m_i}\\
&=\frac{1}{\prod_{i=1}^k\Phi_\omega(q^{-2(i-1)})}\sum_{m_1,\dots,m_k\in\mathbb{Z}}\det[\varphi_\omega(m_i+j)]_{i,j=1}^k\prod_{i=1}^k(q^{-2(i-1)}z_i)^{k+m_i}\\
&=\frac{1}{\prod_{i=1}^k\Phi_\omega(q^{-2(i-1)})}\sum_{m_1>\cdots>m_k}\sum_{\tau\in S(k)}\det[\varphi_\omega(m_{\tau(i)}+j)]_{i,j=1}^k\prod_{i=1}^k(q^{-2(i-1)}z_i)^{k+m_{\tau(i)}}\\
&=\frac{1}{\prod_{i=1}^k\Phi_\omega(q^{-2(i-1)})}\sum_{m_1>\cdots>m_k}\det[\varphi_\omega(m_i+j)]_{i,j=1}^k\det[(q^{-2(i-1)}z_i)^{k+m_j}]_{i,j=1}^k.
\end{align*}
Thus, we have
\[\mathcal{F}(z_1,\dots, z_k)=\sum_{\lambda\in\mathbb{S}_k}P_k(\lambda)\frac{s_\lambda(z_1, q^{-2}z_2,\dots, q^{-2(k-1)}z_k)}{s_\lambda(1, q^{-2},\dots, q^{-2(k-1)})},\]
where
\[P_k(\lambda):=\frac{\det[\varphi_\omega(\lambda_i-i+j)]_{i,j=1}^k}{\prod_{i=1}^k\Phi_\omega(q^{-2(i-1)})}s_\lambda(1, q^{-2},\dots, q^{-2(k-1)})\geq0\]
since each factor is nonnegative. Moreover, substituting $(1,\dots, 1)$ for $(z_1,\dots, z_k)$, we have $\sum_{\lambda\in\mathbb{S}_k}P_k(\lambda)=1$.
\end{proof}

By the above lemma, there exists a quantized character $\chi^N_{q, \omega}\in\mathrm{Ch}(U_q(N))$ such that the corresponding $q^2$-Schur generating functions are given as Equation \eqref{eq:q-char}.
\begin{theorem}\label{thm:q-pubg}
Let $\omega=(\alpha^+,\beta^+,\alpha^-,\beta^-,\gamma^+,\gamma^-)\in\Omega$ satisfying the assumption in Lemma \ref{lem:q-char} and let $\chi^N_{q,\omega}$ be the corresponding quantized character of $U_q(N)$. Then the generator $\mathbb{L}_{\chi^N_{q,\omega}}$ of the Markov semigroup associated with $\chi^N_{q, \omega}$ in Theorem \ref{thm:jun} is given by
\[\mathbb{L}_{\chi^N_{q,\omega}}(\lambda, \mu)=\frac{s_\mu(1, q^{-2},\dots, q^{-2(N-1)})}{s_\lambda(1, q^{-2},\dots, q^{-2(N-1)})}\frac{\det[\varphi_\omega(\mu_j-j-\lambda_i+i)]_{i,j=1}^N}{\prod_{i=1}^N\Phi_\omega(q^{-2(i-1)})}-\delta_{\lambda, \mu}\]
for any $\lambda,\mu\in\mathbb{S}_N$. 
\end{theorem}
\begin{proof}
Remark that the quantum dimension $d_q(\lambda)$ of irreducible corepresentation associated with $\lambda\in\mathbb{S}_N$ coincides with $s_\lambda(q^{N-1}, q^{N-3},\dots, q^{-N+1})$. Since $U_q(N)$ has the same fusion rules of $U(N)$, by Corollary \ref{cor:pubg}, we have
\begin{align*}
&\mathbb{L}_{\chi^N_{q,\omega}}(\lambda, \mu)\\
&=\frac{s_\mu(q^{N-1},q^{N-3},\dots,q^{-N+1})}{s_\lambda(q^{N-1},q^{N-3},\dots,q^{-N+1})}\frac{1}{N!}\int_{\mathbb{T}^N}\mathcal{F}(q^{-N+1}z_1,\dots,q^{N-1}z_N)s_\lambda(z)\overline{s_\mu(z)}|V(z)|^2\prod_{i=1}^N\frac{dz_i}{2\pi\mathrm{i}z_i}-\delta_{\lambda, \mu}.
\end{align*}
By a similar computation in the proof of Theorem \ref{thm:kuma}, we obtain the statement.
\end{proof}

\section{On inductive systems of compact quantum groups}\label{sec:ind}
In this section, we extend the previous results in Section \ref{sec:cqg} to inductive systems of compact quantum groups. We assume that $(G_N)_{N=0}^\infty$ is an inductive system of compact quantum group and $G_0$ is the trivial quantum group. Namely, $C(G_0)=\mathbb{C}$, and $G_N=(C(G_N), \delta_{G_N})$ are compact quantum groups. Moreover, there are the so-called \emph{restriction maps} $\theta_N\colon C(G_N)\to C(G_{N-1})$, i.e., $\theta_N$ is a surjective $*$-homomorphism such that $(\theta_N\otimes\theta_N)\delta_{G_N}=\delta_{G_{N-1}}\theta_N$. We denote by $\mathcal{W}^*(G_N)=(W^*(G_N),\hat\delta_{G_N}, \hat R_{G_N}, \hat\tau^{G_N}, \hat h_{G_N})$ the Woronowicz algebras of group von Neumann algebras $W^*(G_N)$. Let $h_{G_N}$ be the Haar state of $\mathcal{L}^\infty(G_N)$ and $(\pi_{h_{G_N}}, L^2(G_N), \eta_{h_{G_N}})$ the associated GNS-triple. We define the \emph{Kac--Takesaki operator} $\hat W_{G_N}$ as $\hat W_{G_N}:=\sigma_{L^2(G_N)}W_{G_N}^*\sigma_{L^2(G_N)}$, where $W_{G_N}\colon L^2(G_N)\otimes L^2(G_N)\to L^2(G_N)\otimes L^2(G)$ is given by 
\[W_{G_N}\xi\otimes\eta_{h_{G_N}}(a):=(\pi_{h_{G_N}}\otimes\pi_{h_{G_N}})(\delta_{G_N}(a))\xi\otimes\eta_{h_{G_N}}(1)\] 
for any $\xi\in L^2(G_N)$ and $a\in C(G_N)$. By \cite[Lemma 2.10]{Tomatsu}, there exist faithful normal unital $*$-homomorphisms $\Theta_N\colon W^*(G_N)\to W^*(G_{N+1})$ such that $\Theta_N(\pi_{G_N}(\omega))=(\mathrm{id}\otimes\omega\theta_{N+1})(\hat W_{G_N+1})$ for any $\omega\in L^\infty(G_N)_*$. Moreover, we have
\begin{equation}\label{eq:cole}
(\Theta_N\otimes\Theta_N)\hat\delta_{G_N}=\hat\delta_{G_{N+1}}\Theta_N,\quad \Theta_N\hat\tau^{G_N}_t=\hat\tau^{G_{N+1}}_t\Theta_N,\quad\Theta_N\hat R_{G_N}=\hat R_{G_{N+1}}\Theta_N.
\end{equation}
Therefore, we have the inductive limit quantum group $W^*$-algebra $(W^*(G_\infty), \mathfrak{A}, \hat\delta_{G_\infty}, \hat R_{G_\infty}, \hat\tau^{G_\infty})$ in the sense of \cite{Sato3}, that is, 
\begin{itemize}
\item $W^*(G_\infty):=\varinjlim_N (W^*(G_N), \Theta_N)$ is the inductive limit $W^*$-algebra in the sense of Takeda \cite{Takeda},
\item $\mathfrak{A}$ is the Stratila--Voiculescu AF-algebra of the inductive system $(G_N)_{N=0}^\infty$,
\item $\hat\delta_{G_\infty}\colon W^*(G_\infty)\to W^*(G_\infty)\bar\otimes W^*(G_\infty)$ is a comultiplication, that is, unital normal $*$-homomorphism satisfying that $(\mathrm{id}\otimes \hat\delta_{G_\infty})\hat\delta_{G_\infty}=(\hat\delta_{G_\infty}\otimes\mathrm{id})\hat\delta_{G_\infty}$,
\item $\hat R_{G_\infty}$ is a unitary antipode, that is, an involutive normal $*$-anti-automorphism satisfying that $\hat\delta_{G_\infty}\hat R_{G_\infty}=\sigma_{W^*(G_\infty)}(\hat R_{G_\infty}\otimes\hat R_{G_\infty})\hat\delta_{G_\infty}$,
\item $\hat\tau^{G_\infty}=\{\hat\tau^{G_\infty}_t\}_{t\in\mathbb{R}}$ is a deformation automorphism group, that is, a one-parameter automorphism group on $W^*(G_\infty)$ satisfying that $(\hat\tau^{G_\infty}_t\otimes\hat\tau^{G_\infty}_t)\hat\delta_{G_\infty}=\hat\delta_{G_\infty}\hat\tau^{G_\infty}_t$ and $\hat\tau^{G_\infty}_t\hat R_{G_\infty}=\hat R_{G_\infty}\tau^{G_\infty}_t$ for any $t\in\mathbb{R}$.
\end{itemize} 

Remark that there exist normal unital $*$-homomorphism $\Theta_N^\infty\colon W^*(G_N)\to W^*(G_\infty)$ such that $\Theta^\infty_{N+1}\Theta_N=\Theta^\infty_N$ for any $N\geq0$. For a von Neumann algebra $M$, if there are normal unital $*$-homomorphisms $\Phi_N\colon W^*(G_N)\to M$ such that $\Phi_{N+1}\Theta_N=\Phi_N$, then there exists a normal unital $*$-homomorphism $\Phi_\infty\colon W^*(G_\infty)\to M$ such that $\Phi_\infty\Theta^\infty_N=\Phi_N$ for any $N\geq0$.

We define a \emph{quantized character} of $G_\infty$ as a normal $\hat\tau^{G_\infty}$-KMS state (see \cite{Sato3}). We denote by $\mathrm{Ch}(G_\infty)$ the set of all quantized characters of $G_\infty$. Then we equip $\mathrm{Ch}(G_\infty)$ with the topology of point-wise convergence on the Stratila--Voiculescu AF-algebra $\mathfrak{A}$. Let $\mathcal{E}:=\mathrm{ex}(\mathrm{Ch}(G_\infty))$ equipped with the relative topology induced by $\mathrm{Ch}(G_\infty)$. We recall that $\mathrm{Ch}(G_\infty)$ is a simplex, that is, there exists a one-to-one correspondence $\chi\in\mathrm{Ch}(G_\infty)\mapsto P_\chi\in\mathcal{M}_p(\mathcal{E})$ such that $\chi=\int_{\mathcal{E}}\chi_\omega dP_\chi(\chi_\omega)$. See \cite[Corollary 3.1]{Sato3}.

For any $\omega\in L^\infty(G_N)_*$ we have $\hat\epsilon_{G_{N+1}}(\Theta_N(\pi_{G_N}(\omega)))=\omega(\theta_{N+1}(1))=\omega(1)=\hat\epsilon_{G_N}(\pi_{G_N}(\omega))$. Thus, there exits a normal $*$-homomorphism $\hat\epsilon_\infty\colon W^*(G_\infty)\to\mathbb{C}$ such that $\hat\epsilon_\infty\Theta^\infty_N=\hat\epsilon_{G_N}$ for any $N\geq0$. We have $(\mathrm{id}\otimes\hat\epsilon_\infty)\hat\delta_{G_\infty}=(\hat\epsilon_\infty\otimes\mathrm{id})\hat\delta_{G_\infty}=\mathrm{id}$ since the equation holds true on $\bigcup_{N\geq0}\Theta^\infty_N(W^*(G_N))$ and it is $\sigma$-weakly dence in $W^*(G_\infty)$. In this section, for any $\chi\in\mathrm{Ch}(G_\infty)$ we study the dynamics $((\mathrm{id}\otimes\varphi_t)\hat\delta_{G_\infty})_{t\geq0}$, where $\varphi_t=\exp_*(t(\chi-\hat\epsilon_\infty))$.

\begin{lemma}
Let $\chi\in\mathrm{Ch}(G_\infty)$. Then $\varphi_t=\exp_*(t(\chi-\hat\epsilon_\infty))\in\mathrm{Ch}(G_\infty)$.
\end{lemma}
\begin{proof}
By Lemma \ref{lem:yosa}, $\varphi_t\Theta^\infty_N=\exp_*(t(\chi\Theta^\infty_N-\hat\epsilon_{G_N}))\in\mathrm{Ch}(G_N)$ since $\chi\Theta^\infty_N\in\mathrm{Ch}(G_N)$. Thus, there exists a unique quantized character $\varphi\in\mathrm{Ch}(G_\infty)$ such that $\varphi\Theta^\infty_N=\varphi_t\Theta^\infty_N$ for any $N\geq0$. See \cite{Sato1}, \cite{Sato3}. Since $\bigcup_{N\geq0}\Theta^\infty_N(W^*(G_N))$ is $\sigma$-weakly dense in $W^*(G_\infty)$, we have $\varphi=\varphi_t$, i.e., $\varphi_t\in\mathrm{Ch}(G_\infty)$.
\end{proof}

Now we define our main object in this section.
\begin{definition}
By the above lemma and \cite[Lemma 3.2]{Sato3}, for $\chi\in\mathrm{Ch}(G_\infty)$ we obtain the one-parameter semigroup on $\mathrm{Ch}(G_\infty)$ by $\varphi\mapsto\varphi*\varphi_t$, where $\varphi_t=\exp_*(t(\chi-\hat\epsilon_\infty))$ for any $t\geq0$. Then we have the corresponding semigroup $(Q^\chi_t)_{t\geq0}$ on $\mathcal{M}_p(\mathcal{E})$ such that $Q^\chi_t(P_\varphi)=P_{\varphi*\varphi_t}$, where $P_\varphi, P_{\varphi*\varphi_t}$ are the Borel probability measures corresponding to $\varphi, \varphi*\varphi_t\in\mathrm{Ch}(G_\infty)$, respectively. 
\end{definition}
The purpose of this section is studying the above dynamics $(Q^\chi_t)_{t\geq0}$, and we particularly show the dynamics $(Q^\chi_t)_{t\geq0}$ coincide with the ``projective limits" of Markov dynamics on the $\widehat{G_N}$. See Proposition \ref{prop:contact}.

\begin{lemma}
Let $\chi\in\mathrm{Ch}(G_\infty)$, $\varphi_t=\exp_*(t(\chi-\hat\epsilon_\infty))$ and $T_t=(\mathrm{id}\otimes\varphi_t)\hat\delta_{G_\infty}$. Then $T_t$ preserves $\Theta^\infty_N(W^*(G_N))$ and $\Theta^\infty_N(Z(W^*(G_N)))$ for any $N\geq0$.
\end{lemma}
\begin{proof}
We have $T_t\Theta^\infty_N=\Theta^\infty_N(\mathrm{id}\otimes\varphi_t\Theta^\infty_N)\hat\delta_{G_N}$ and $\varphi_t\Theta^\infty_N=\exp_*(t(\chi\Theta^\infty_N-\hat\epsilon_{G_N}))$. Since $\chi\Theta^\infty_N\in\mathrm{Ch}(G_N)$, the linear map $(\mathrm{id}\otimes\varphi_t\Theta^\infty_N)\hat\delta_{G_N}$ preserves $W^*(G_N)$ and $Z(W^*(G_N))$.
\end{proof}

Let $V_{G_N}$ be the multiplicative unitary of $\mathcal{L}^\infty(G_N)$. We define $\alpha_0^{G_N}(x):=V_{G_N}(x\otimes 1)V_{G_N}^*$ and $E_{G_N}(x):=(\mathrm{id}\otimes h_{G_N})(\alpha_0^{G_N}(x))$ for any $x\in W^*(G_N)$. Recall that $\alpha_0^{G_N}$ is a right action of $\mathcal{L}^\infty(G_N)$ on $W^*(G_N)$, and $E_{G_N}$ is a normal conditional expectation from $W^*(G_N)$ onto $Z(W^*(G_N))$ (see Lemma \ref{lem:cond}). By \cite[Theorem 1.4.2]{NeshveyevTuset} and Corollary \ref{cor:hinata}, for any $x=\Phi_G^{-1}((A_\alpha)_{\alpha\in\widehat{G_N}})$ we have 
\[E_{G_N}(x)=\sum_{\alpha\in\widehat{G_N}}\frac{1}{d_q(\alpha)}\sum_{i,j=1}^{\dim H_\alpha}\Phi_{G_N}^{-1}(e^\alpha_{ij}F_\alpha Ae^\alpha_{ji})=\sum_{\alpha\in\widehat{G_N}}\chi_\alpha^0(A)\Phi_{G_N}^{-1}(1_{H_\alpha}),\]
where $(e^\alpha_{ij})_{i,j=1}^{\dim H_\alpha}$ is a matrix unit system of $B(H_\alpha)$ and $1_{H_\alpha}$ is the identity map on $H_\alpha$. See Remark \ref{rem:q-aina} for the definition of $\chi_\alpha^0$.

Here we refer to the paper by Ueda \cite{Ueda20}. He reformulated the spherical representation theory (due to Olshanski \cite{Olshanski90}, \cite{Olshanski03} ... etc) in the setting of $C^*$-algebras with flows. Then, by replacing $\widehat{G_N}$ with the set of minimal projections of the center of $W^*(G_N)$, we obtain the same formula of $E_{G_N}$. Moreover, the following investigation of \emph{links} of the $\widehat G_N$ also extends to the general setting (see \cite[Section 7, 8]{Ueda20}): Recall that $Z(W^*(G_N))$ and $\ell^\infty(\widehat{G_N})$ are $*$-isomorphic. Thus, we obtain Markov kernels $\Lambda_N^{N+1}$ from $\widehat{G_{N+1}}$ to $\widehat{G_N}$ by $E_{G_{N+1}}\Theta_N|_{W^*(G_N)}\colon Z(W^*(G_N))\to Z(W^*(G_{N+1}))$. Then $\Lambda_N^{N+1}$ induces the map $\mathcal{M}_p(\widehat{G_{N+1}})\to\mathcal{M}_p(\widehat{G_N})$, and $\varprojlim_N\mathcal{M}_p(\widehat{G_N})$ is nothing but the set of so-called \emph{coherent systems}. We equip $\varprojlim_N\mathcal{M}_p(\widehat G_N)$ with the topology of component-wise weakly convergence. Then we have affine homeomorphism between $\varprojlim_N\mathcal{M}_p(\widehat G_N)$ and $\mathcal{M}_p(\mathcal{E})$. See \cite[Proposition 7.10, Corollary 8.3(3)]{Ueda20}. Namely $\mathcal{E}$ is the \emph{boundary} of $(\widehat{G_N})_{N=1}^\infty$ with the links $(\Lambda_N^{N+1})_{N=1}^\infty$ in the sense of Boridin--Olshanski \cite{BO}. Then for any $N\geq1$ we obtain the Markov kernel $\Lambda^\infty_N$ from $\mathcal{E}$ to $\widehat{G_N}$ such that for any $P\in\mathcal{M}_p(\mathcal{E})$ the corresponding coherent system is $(P\Lambda^\infty_N)_{N\geq1}$.

Recall that for any $\psi\in W^*(G_N)_*$ we define $Q_\psi:=(\mathrm{id}\otimes\psi)\hat\delta_{G_N}|_{Z(W^*(G_N))}$.
\begin{lemma}\label{lem:hayabusa}
Let $\psi_{N+1}\in W^*(G_{N+1})_*$ be a $\alpha_0^{G_{N+1}}$-invariant functional and $\psi_N:=\psi_{N+1}\Theta_N$. Then $\Psi_NQ_{\psi_N}=Q_{\psi_{N+1}} \Psi_N$.
\end{lemma}
\begin{proof}
By $\alpha^{G_{N+1}}_0$-invariance of $\psi_{N+1}$ and Equation \eqref{eq:action}, for any $x\in Z(W^*(G_N))$ we have
\begin{align*}
[Q_{\psi_{N+1}}\Psi_N](x)
&=(\mathrm{id}\otimes\psi_{N+1}\otimes h_{G_{N+1}})((\hat\delta_{G_{N+1}}\otimes\mathrm{id})(\alpha^{G_{N+1}}_0(\Theta_N(x))))\\
&=(\mathrm{id}\otimes\psi_{N+1}\otimes h_{G_{N+1}})(\mathrm{Ad}V_{G_{N+1}13}((\mathrm{id}\otimes\alpha^{G_{N+1}}_0)(\hat\delta_{G_{N+1}}(\Theta_N(x)))))\\
&=(\mathrm{id}\otimes\psi_{N+1}\otimes h_{G_{N+1}})(\mathrm{Ad}V_{G_{N+1}13}(\hat\delta_{G_{N+1}}(\Theta_N(x))\otimes1)).
\end{align*}
Since $V_{G_{N+1}13}=\sigma_{L^2(G_{N+1})23}V_{G_{N+1}12}\sigma_{L^2(G_{N+1})23}$, we have
\begin{align*}
[\Psi_NQ_{\psi_N}](x)
&=(\mathrm{id}\otimes h_{G_{N+1}}\otimes \psi_{N+1})((\alpha^{G_{N+1}}_0\otimes\mathrm{id})(\hat\delta_{G_{N+1}}(\Theta_N(x))))\\
&=(\mathrm{id}\otimes h_{G_{N+1}}\otimes \psi_{N+1})(\mathrm{Ad}V_{G_{N+1}12}\sigma_{L^2(G_{N+1})23}(\hat\delta_{G_{N+1}}(\Theta_N(x))\otimes1))\\
&=(\mathrm{id}\otimes\psi_{N+1}\otimes h_{G_{N+1}})(\mathrm{Ad}V_{G_{N+1}13}(\hat\delta_{G_{N+1}}(\Theta_N(x))\otimes1)).
\end{align*}
Therefore, we have $\Psi_NQ_{\psi_N}=Q_{\psi_{N+1}} \Psi_N$. 
\end{proof}

\begin{proposition}
Let $\chi\in\mathrm{Ch}(G_\infty)$, $\psi_N:=\chi\Theta^\infty_N-\hat\epsilon_{G_N}$ and $Q^N_t:=Q_{\exp_*(t\psi_N)}$ for any $t\geq0$ and $N\geq1$. Then $\Psi_N Q^N_t=Q^{N+1}_t\Psi_N$ for any $N\geq0$.
\end{proposition}
\begin{proof}
Remark that $\chi\Theta^\infty_N\in\mathrm{Ch}(G_N)$ (see \cite[Section 3]{Sato3}), and hence $\psi_N$ is $\alpha^N_0$-invariant by Lemma \ref{lem:char_inv}, \ref{lem:counit_inv}. Since $\hat\epsilon_{G_{N+1}}\Theta_N=\hat\epsilon_{G_N}$, we have $\psi_{N+1}\Theta_N=\psi_N$ and $\exp_*(t\psi_{N+1})\Theta_N=\exp_*(t\psi_N)$ for any $t\geq0$. Since $\exp_*(t\psi_{N+1})$ is $\alpha^{G_{N+1}}_0$-invariant, $\Psi_N Q^N_t=Q^{N+1}_t\Psi_N$ by Lemma \ref{lem:hayabusa}.
\end{proof}

Let $Q^N_t$ be the same as in the above. Then, by \cite[Proposition 2.4]{BO}, we obtain a unique Markov semigroup $(Q^\infty_t)_{t\geq0}$ on $\mathcal{E}$ such that $Q^\infty_t\Lambda^\infty_N=\Lambda^\infty_N Q^N_t$ for any $t\geq0$ and $N\geq1$. Recall that we also obtained the semigroup $(Q^\chi_t)_{t\geq0}$ on $\mathcal{M}_p(\mathcal{E})$ such that $P_\varphi Q^\chi_t=P_{\varphi*\varphi_t}$ for any $\varphi\in\mathrm{Ch}(G_\infty)$, where $\varphi_t=\exp_*(t(\chi-\hat\epsilon_\infty))$. The following is the conclusion of this section.
\begin{proposition}\label{prop:contact}
Let $\chi\in\mathrm{Ch}(G_\infty)$ and $(Q^\infty_t)_{t\geq0}$ the corresponding Markov semigroup on $\mathcal{E}$. Then $(Q^\infty_t)_{t\geq0}$ coincides with $(Q^\chi_t)_{t\geq0}$, that is, $PQ^\infty_t=PQ^\chi_t$ for any $P\in\mathcal{M}_p(\mathcal{E})$ and $t\geq0$.
\end{proposition}
\begin{proof}
We may assume that $P=P_\varphi$ for some $\varphi\in\mathrm{Ch}(G_\infty)$. Thus, it suffices to show that $P_\varphi Q^\infty_t=P_{\varphi*\varphi_t}$ for any $t\geq0$. Since $P_\varphi Q^\infty_t\Lambda^\infty_N=P_{(\varphi*\varphi_t)\Theta^\infty_N}$ and $(P_{(\varphi*\varphi_t)\Theta^\infty_N})_{N\geq1}$ is the coherent system corresponding to $P_{\varphi*\varphi_t}$, we have $P_\varphi Q^\infty_t=P_{\varphi*\varphi_t}$ for any $t\geq0$.
\end{proof}

\section{Markov semigroups on the Feller boundary of the Gelfand--Tsetlin graph}\label{sec:GT}
In this section, we return to the case of unitary groups. Recall that $\mathrm{ex}(\mathrm{Ch}(U(\infty)))$ is completely parametrized by $\Omega$ in Section \ref{sec:unitary}. We remark that the set of quantized characters is homeomorphic to the set of ordinary characters of $U(\infty)$ with the topology of uniform convergence on compact subsets. Since any compact subsets in $U(\infty)$ is contained in $U(N)$ for some $N\geq0$ (see \cite[Proposition 6.5(i)]{HSTH}), this topology is equivalent to the topology of uniform convergence on each $U(N)$ for $N\geq0$. On the other hand, we endow $\Omega$ with the relative topology induced by the product topology of $\mathbb{R}_{\geq0}^\infty\times\mathbb{R}_{\geq0}^\infty\times\mathbb{R}_{\geq0}^\infty\times\mathbb{R}_{\geq0}^\infty\times\mathbb{R}_{\geq0}\times\mathbb{R}_{\geq0}$. Then, by \cite[Theorem 8.1]{Olshanski03}, $\mathrm{ex}(\mathrm{Ch}(U(\infty)))$ and $\Omega$ are homeomorphic. Thus, in what follows, we identify $\Omega$ with $\mathrm{ex}(\mathrm{Ch}(U(\infty)))$.

We write $\lambda\prec\mu$ if 
$\lambda_1\geq\mu_1\geq\lambda_2\geq\cdots\geq\mu_N\geq\lambda_{N+1}$
for any $\mu\in\mathbb{S}_N$ and $\lambda\in\mathbb{S}_{N+1}$. Then we have the Markov kernel $\Lambda^{N+1}_N$ from $\widehat{U(N+1)}\cong\mathbb{S}_{N+1}$ to $\widehat{U(N)}\cong\mathbb{S}_N$ in Section \ref{sec:ind} is geven by
\[\Lambda^{N+1}_N(\lambda, \mu)=\begin{cases}\frac{s_\mu(1,\dots, 1)}{s_\lambda(1,\dots, 1)}&\mu\prec\lambda,\\0&\text{otherwise},\end{cases}\]
where $s_\mu(1,\dots, 1)$ and $s_\lambda(1,\dots,1)$ are specializations of Schur polynomials with $N-1$ and $N$ variables, respectively. Moreover, the Markov kernel $\Lambda^\infty_N\colon\Omega\times\mathbb{S}_N\to[0,1]$ is given by
\[\Lambda^\infty_N(\omega, \lambda)=s_\lambda(1,\dots, 1)\det[\varphi_\omega(\lambda_i-i+j)]_{i,j=1}^N\]
for any $\omega\in \Omega$ and $\lambda\in\mathbb{S}_N$, where the $\varphi_\omega(n)$ is the coefficient of $z^n$ in the Laurent expansion of $\Phi_\omega(z)$. See \cite{BO} for more details. Then, by \cite[Proposition 3.3, 3.4]{BO}, the boundary $\Omega$ of the sequence $(\mathbb{S}_ N)_{N=1}^\infty$ with the Markov kernels $(\Lambda^{N+1}_N)_{N=1}^\infty$ is Feller, that is, the Markov kernels $\Lambda^{N+1}_N$ and $\Lambda^\infty_N$ gives mappings $c_0(\mathbb{S}_N)\to c_0(\mathbb{S}_{N+1})$ and $c_0(\mathbb{S}_N)\to c_0(\Omega)$ by
\[\Lambda^{N+1}_Nf(\lambda)=\sum_{\mu\in\mathbb{S}_N}\Lambda^{N+1}_N(\lambda, \mu)f(\mu),\quad \Lambda^\infty_Nf(\omega):=\sum_{\mu\in\mathbb{S}_N}\Lambda^\infty_N(\omega,\mu)f(\mu)\]
for any $\lambda\in\mathbb{S}_{N+1}$, $\omega\in\Omega$, and $f\in c_0(\mathbb{S}_N)$. 
\begin{proposition}\label{prop:GT}
For any $\chi\in\mathrm{Ch}(U(\infty))$ the Markov semigroup $(Q^\chi_t)_{t\geq0}$ in Proposition \ref{prop:mar_bry} is Feller. Moreover, its generator $L^\chi$ is determined by $L^\chi\Lambda^\infty_Nf=\Lambda^\infty_NL^\chi_Nf$ for any $f\in c_0(\mathbb{S}_N)$ and $N\geq1$, where $L^\chi_N$ is the generator of $(Q^N_t)_{t\geq0}$.
\end{proposition}
\begin{proof}
By \cite[Proposition 2.4]{BO}, the Markov semigroup $(Q^\chi_t)_{t\geq0}$ is Feller. Recall that $L^\chi_N=Q^{\chi\Theta^\infty_N}-\mathrm{id}$ (see the proof of Theorem \ref{thm:jun}). Thus, we have $\|L^\chi_N\|\leq2$ for any $N\geq1$. By \cite[Lemma 2.3]{BO}, $\bigcup_{N\geq1}\Lambda^\infty_N(c_0(\mathbb{S}_N))\subset c_0(\Omega)$ is dense. Thus, $L^\chi$ is determined by the relations $L^\chi\Lambda^\infty_Nf=\Lambda^\infty_NL^\chi_Nf$ for any $f\in c_0(\mathbb{S}_N)$ and $N\geq1$.
\end{proof}

\section{Discrete-time dynamics on the Gelfand--Tsetlin patterns}\label{sec:discrete-time}
In the rest of the paper, we discuss discrete-time Markov dynamics on the set of Gelfand--Tsetlin patterns generated by quantized characters of $U_q(N)$ in Lemma \ref{lem:q-char}. Recall that every quantized character $\chi\in\mathrm{Ch}(U_q(N))$ gives the Markov operator $\mathbb{Q}^{\chi}:=(\mathrm{id}\otimes\chi)\hat\delta_{U_q(N)}|_{Z(W^*(U_q(N)))}$ on $Z(W^*(U_q(N)))\cong\ell^\infty(\mathbb{S}_N)$. Moreover, they are intertwining by the Markov kernels $\Lambda^{N+1}_N$ in Section \ref{sec:ind}. If $G_N=U_q(N)$, we can describe the kernel $\Lambda^{N+1}_N$ explicitly by
\[\Lambda^N_{N-1}(\lambda, \mu)=\begin{cases}q^{N|\mu|-(N-1)|\lambda|}\frac{s_\mu(q^{N-2},q^{N-4},\dots,q^{-N+2})}{s_\lambda(q^{N-1},q^{N-3},\dots, q^{-N+1})}&\mu\prec\lambda,\\0&\text{otherwise}.\end{cases}\]
See \cite[Section 3]{Sato1}. 

\begin{remark}
The construction of $\mathbb{Q}^{\chi}$ is the generalization of Markov chains due to Kuan \cite{Kuan18}. Let $G$ be a compact quantum group and $\chi$ its quantized character. By a similar proof of Corollary \ref{cor:pubg}, for any $\alpha, \beta\in\widehat G$ we have 
\[\mathbb{Q}^\chi(\alpha, \beta)=\frac{d_q(\beta)}{d_q(\alpha)}\sum_{\gamma\in\widehat{G}}\frac{P_\chi(\gamma)N^\beta_{\alpha,\gamma}}{d_q(\gamma)}\]
if the irreducible decomposition of $\chi$ is $\sum_{\gamma\in\widehat{G}}P_\chi(\gamma)\chi_\gamma$. In particular, when $G=U_q(N)$ and $\chi^N_{q,\omega}$ is its quantized character in Lemma \ref{lem:q-char}, for any $\lambda, \mu\in\mathbb{S}_N$ we have
\begin{align*}
\mathbb{Q}^{\chi^N_{q,\omega}}(\lambda,\mu)=\frac{s_\mu(q^{N-1},q^{N-3},\dots, q^{-N+1})}{s_\lambda(q^{N-1}, q^{N-3},\dots, q^{-N+1})}\frac{1}{N!}\int_{\mathbb{T}^N}\prod_{i=1}^N\frac{\Phi_\omega(q^{-2(i-1)z_i})}{\Phi_\omega(q^{-2(i-1)})}s_\lambda(z)\overline{s_\mu(z)}|V(z)|^2dz.
\end{align*}
Therefore, if $\mathcal{S}(z_1,\dots,z_N)$ is the $q^2$-Schur generating function of a probability measure $P$ on $\mathbb{S}_N$, then the $q^2$-Schur generating function of $P\mathbb{Q}^\chi$ is given as 
\[\mathcal{S}(z_1,\dots,z_N)\prod_{i=1}^N\frac{\Phi_\omega(q^{-2(i-1)}z_i)}{\Phi_\omega(q^{-2(i-1)})}.\]
The dynamics above is discussed in \cite[Section 2]{BG}, i.e., we have given a representation-theoretic construction of the dynamics in \cite{BG}.
\end{remark}

We also discuss the relationship to the dynamics in \cite{BF}. First we recall the Toeplitz-like transition probabilities in \cite{BF}. Let $\alpha_1,\dots, \alpha_N$ be nonzero complex numbers and $F$ an analytic function in an annulus $A$ centered at the origin that contains $\alpha_1^{-1},\dots, \alpha_N^{-1}$. For $n=1,\dots, N$ we denote $\mathfrak{X}_n=\{(x_1<\cdots<x_n)\in\mathbb{Z}^n\}$. Assume that $F(\alpha_1^{-1})\cdots F(\alpha_N^{-1})\neq0$ and $\det[\alpha_i^{x_j}]_{i,j=1}^n/\det[\alpha_i^{j-1}]_{i,j=1}^n\neq0$ for any $X=(x_1,\dots,x_n)\in\mathfrak{X}_n$. We also denote 
\[f(m)=\frac{1}{2\pi\mathrm{i}}\oint F(z)\frac{dz}{z^{m+1}}\]
for any $m\in\mathbb{Z}$, where the integral is taken over any positively oriented simple loop in $A$. Then we define kernels on $\mathfrak{X}_n\times \mathfrak{X}_n$ and $\mathfrak{X}_n\times\mathfrak{X}_{n-1}$ by
\[T_n(\alpha_1,\dots,\alpha_n; F)(X, Y)=\frac{\det[\alpha_i^{y_j}]_{i, j=1}^n}{\det[\alpha_i^{x_j}]_{i, j=1}^n}\frac{\det[f(x_i-y_j)]_{i,j=1}^n}{\prod_{i=1}^nF(\alpha_i^{-1})} \quad \text{for any }X, Y\in\mathfrak{X}_n,\]
\[T^n_{n-1}(\alpha_1,\dots,\alpha_n; F)(X, Y)=\frac{\det[\alpha_i^{y_j}]_{i, j=1}^{n-1}}{\det[\alpha_i^{x_j}]_{i, j=1}^n}\frac{\det[f(x_i-y_j)]_{i,j=1}^n}{\prod_{i=1}^{n-1}F(\alpha_i^{-1})}\quad \text{for any }X\in\mathfrak{X}_n\, Y\in\mathfrak{X}_{n-1},\]
where $y_n=\text{virt}$ is \emph{virtual} variable and $f(x_i-\text{virt})$ is defined suitably. See \cite[Section 2]{BF}. 

We show that $Q^{\chi_N}$ and $\Lambda^{N+1}_N$ are given as the above forms. In what follows, we use new coordinates of $\mathbb{S}_n$ given by $x_k(\lambda)=\lambda_{n-k+1}-n+k-1$ for $k=1,\dots, n$ and $\lambda\in\mathbb{S}_n$. Then $X_n(\lambda):=(x_1(\lambda),\dots, x_n(\lambda))\in\mathfrak{X}_n$. Moreover, $\mu\prec\lambda$ for $\mu\in\mathbb{S}_n$ and $\mu\in\mathbb{S}_{n-1}$ if and only if 
$x_1(\lambda)<x_1(\mu)\leq x_2(\lambda)<\dots<x_{n-1}(\mu)\leq x_n(\lambda)$.
Thus, we write $Y\prec X$ if this condition holds true for $X\in\mathfrak{X}_N$ and $Y\in\mathfrak{X}_{n-1}$.

\begin{theorem}\label{thm:d-dynamics}
Let $\omega\in\Omega$ as in Lemma \ref{lem:q-char} and $\chi_n:=\chi_{q,\omega}\Theta^\infty_n$. Then
\[\mathbb{Q}^{\chi_n}(\lambda, \mu)=T_n(1,q^{-2},\dots, q^{-2(n-1)}; \Psi_\omega)(X_n(\lambda),X_n(\mu))\]
for any $\lambda, \mu\in\mathbb{S}_n$, where $\Psi_\omega(z):=\Phi(z^{-1})$.
\end{theorem}
\begin{proof}
By a similar computation in Thenrem \ref{thm:kuma}, for any $\lambda, \mu\in\mathbb{S}_n$ we have
\[\mathbb{Q}^{\chi_n}(\lambda, \mu)=\frac{s_\mu(1, q^{-2},\dots, q^{-2(n-1)})}{s_\lambda(1, q^{-2},\dots, q^{-2(n-1)})}\frac{\det[\varphi_\omega(\mu_j-j-\lambda_i+i)]_{i,j=1}^n}{\prod_{i=1}^n\Phi_\omega(q^{-2(i-1)})}.\]
We remark that the coefficients of $z^m$ in the Laurent expansion of $\Psi_\omega$ coincide with $\varphi_\omega(-m)$. Therefore,
\begin{align*}
\mathbb{Q}^{\chi_n}(\lambda, \mu)
&=\frac{\det[q^{-2(i-1)(\mu_j+n-j)}]_{i, j=1}^n}{\det[q^{-2(i-1)(\lambda_j+n-j)}]_{i, j=1}^n}\frac{\det[\varphi_\omega(\mu_j-j-\lambda_i+i)]_{i,j=1}^n}{\prod_{i=1}^n\Phi_\omega(q^{-2(i-1)})}\\
&=\frac{\det[q^{-2(i-1)x_{n-j+1}(\mu)}]_{i, j=1}^n}{\det[q^{-2(i-1)x_{n-j+1}(\lambda)}]_{i, j=1}^n}\frac{\det[\varphi_\omega(x_{n-j+1}(\mu)-x_{n-i+1}(\lambda))]_{i,j=1}^n}{\prod_{i=1}^n\Psi_\omega(q^{2(i-1)})}\\
&=T_n(1,q^{-2},\dots, q^{-2(n-1)}; \Psi_\omega)(X_n(\lambda),X_n(\mu)).
\end{align*}
\end{proof}

\begin{proposition}
Let $F_n(z)=1/(1-q^{-2(n-1)}z)$. Then 
\[\Lambda^n_{n-1}(\lambda, \mu)=T^n_{n-1}(1, q^{-2},\dots, q^{-2(n-1)};F_n)(X_n(\lambda), X_{n-1}(\mu))\]
for any $\lambda\in\mathbb{S}_n$ and $\mu\in\mathbb{S}_{n-1}$. Remark that 
\[f(m):=\frac{1}{2\pi\mathrm{i}}\oint F(z)\frac{dz}{z^{m+1}}=\begin{cases}q^{-2(n-1)m}&m\geq0,\\0&m<0.\end{cases}\]
Moreover, set $f(m-\text{virt})=q^{-2(n-1)m}$.
\end{proposition}
\begin{proof}
If $\mu\prec\lambda$, then we have
\begin{align*}
\Lambda^n_{n-1}(\lambda, \mu)
&=q^{n|\mu|-(n-1)|\lambda|}\frac{s_\mu(q^{n-2},q^{n-4},\dots,q^{-n+2})}{s_\lambda(q^{n-1},q^{n-3},\dots, q^{-n+1})}\\
&=q^{2(n-1)(|\mu|-|\lambda|)}\frac{\det[q^{-2(i-1)(\mu_j+n-1-j)}]_{i, j=1}^{n-1}}{\det[q^{-2(i-1)(\lambda_i+n-j)}]_{i, j=1}^n}\prod_{i=1}^{n-1}(q^{-2(i-1)}-q^{-2(n-1)}).
\end{align*}
By \cite[Lemma 2.13(ii)]{BF}, if $X_{n-1}(\mu)\prec X_n(\lambda)$, we have
\begin{align*}
&T^n_{n-1}(1,q^{-2},\dots, q^{-2(n-1)}; F)(X_n(\lambda),X_{n-1}(\mu))\\
&=\frac{\det[q^{-2(i-1)y_j}]_{i,j=1}^{n-1}}{\det[q^{-2(i-1)x_j}]_{i,j=1}^n}\frac{(-1)^{n-1}q^{-2(n-1)(\sum_{i=1}^nx_i-\sum_{i=1}^{n-1}y_i)}}{\prod_{i=1}^{n-1}(1-q^{-2(n-1)}q^{2(i-1)})^{-1}}\\
&=q^{2(n-1)(|\mu|-|\lambda|)}\frac{\det[q^{-2(i-1)(\mu_j+n-1-j)}]_{i, j=1}^{n-1}}{\det[q^{-2(i-1)(\lambda_i+n-j)}]_{i, j=1}^n}\prod_{i=1}^{n-1}(q^{-2(i-1)}-q^{-2(n-1)}).
\end{align*}
Otherwise, $T_n(1,q^{-2},\dots, q^{-2(n-1)}; F_n)(X_n(\lambda),X_{n-1}(\mu))=0$. Namely, we have
\[\Lambda^n_{n-1}(\lambda, \mu)=T_n(1,q^{-2},\dots, q^{-2(n-1)}; F_n)(X_n(\lambda),X_{n-1}(\mu)).\]
\end{proof}

By Lemma \ref{lem:hayabusa}, we have $\Delta^n_{n-1}:=\mathbb{Q}^{\chi_n}\Lambda^n_{n-1}=\Lambda^n_{n-1}\mathbb{Q}^{\chi_{n-1}}$. Then, by the above two lemmas and \cite[Proposition 2.10]{BF}, we obtain
$\Delta^n_{n-1}=T^n_{n-1}(1, q^{-2},\dots, q^{-2(n-1)};\Psi_\omega F_n)$. 
Let
\begin{align*}
\mathcal{S}_N
&:=\{\underline{X}=(X_1\prec\cdots\prec X_N)\in\mathfrak{X}_1\times\cdots\times \mathfrak{X}_N\}\\
&=\left\{\underline{X}=(X_1,\dots, X_N)\in\mathfrak{X}_1\times\cdots\times \mathfrak{X}_N\,\middle|\,\prod_{n=2}^n\Lambda^n_{n-1}(X_n, X_{n-1})\neq0\right\}.
\end{align*}
We remark that we can identify $\mathcal{S}_N$ with the set of Gelfand--Tsetlin patterns of lenght $N$. Then we define a Markov operator $P_N$ on $\mathcal{S}_N$ by
\[P_N(\underline{X}, \underline{Y}):=\begin{cases}\mathbb{Q}^{\chi_1}(\lambda^{(1)}, \mu^{(1)})\prod_{n=2}^N\frac{\mathbb{Q}^{\chi_n}(\lambda^{(n)}, \mu^{(n)})\Lambda^n_{n-1}(\mu^{(n)}, \mu^{(n-1)})}{\Delta^n_{n-1}(\lambda^{(n)},\mu^{(n-1)})}&\prod_{n=2}^N\Delta^n_{n-1}(\lambda^{(n)},\mu^{(n-1)})>0,\\0&\text{otherwise},\end{cases}\]
where $X_n=X_n(\lambda^{(n)})$ and $Y_n=X_n(\mu^{(n)})$. Then the dynamics generated by $P_N$ involve several particle systems as appropriate projections. See \cite{BF} for more details. Kuan \cite{Kuan18} also discussed the relationship between the above dynamics and the representation theory of $U(n)$. In this paper, we extended this relationship to the quantum unitary groups $U_q(n)$.

\appendix
\section{Fourier analysis on compact quantum groups}\label{app:FA}
The contents in this section may be well known for experts. However, we provide this section to avoid any inconvenience in different notations depending on papers. Let $G=(C(G), \delta_G)$ be a compact quantum group and $h_G$ its Haar state. We denote by $(\pi_{h_G}, L^2(G), \eta_{h_G})$ the GNS-triple associated with $h_G$. We assume that $h_G$ is faithful. Thus, $\eta_G\colon C(G)\to L^2(G)$ is injective. Let $\widehat G$ be the set of all equivalence classes of irreducible corepresentation of $G$. For any $\alpha\in\widehat G$ we fix a representatives $(U_\alpha, H_\alpha)$ in $\alpha$. Let $\{f^G_z\}_{z\in\mathbb{C}}$ the Woronowicz character of $(C(G), \delta_G)$ and define $F_\alpha^z:=(\mathrm{id}\otimes f^G_z)(U_\alpha)$, $d_q(\alpha):=\mathrm{Tr}(F_\alpha)$. Since the matrices $F_\alpha^z$ ($z\in\mathbb{C}$) are normal and commute with each otherr, we may choose a matrix unit system $(e^\alpha_{ij})_{i. j=1}^{\dim (H_\alpha)}$ of $B(H_\alpha)$ such that $F_\alpha^z$ are diagonal matrices. Then we define $u^\alpha_{ij}\in C(G)$ by $U_\alpha=\sum_{i, j=1}^{\dim(H_\alpha)}e^\alpha_{ij}\otimes u^\alpha_{ij}$. For any $\xi\in L^2(G)$ we define $\omega_\xi\in C(G)^*$ by $\omega_\xi(a):=\langle\eta_G(a), \xi\rangle$ for any $a\in C(G)$. We can show the following Peter--Weyl type theorem:

\begin{theorem}\label{thm:PW}
For any $\xi\in L^2(G)$ define $\hat\xi(\alpha):=(\mathrm{id}\otimes\omega_\xi)(U_\alpha)^*$. Then the mapping 
\[\mathcal{F}_G\colon\xi\in L^2(G)\mapsto \left(\sqrt{d_q(\alpha)}\hat\xi(\alpha)\right)_{\alpha\in\widehat G}\in\ell^2\mathchar`-\bigoplus_{\alpha\in\widehat G}B(H_\alpha)\] 
is an isometric isomorphism of Hilbert spaces, where $\ell^2\mathchar`-\bigoplus_{\alpha\in\widehat G}B(H_\alpha)$ is the completion of the direct sum $\bigoplus_{\alpha\in\widehat G}B(H_\alpha)$ with respect to the norm defined by the following inner product:
\[\langle (A_\alpha)_{\alpha\in\widehat G}, (B_\alpha)_{\alpha\in\widehat G}\rangle:=\sum_{\alpha\in\widehat G}\mathrm{Tr}_{H_\alpha}(F_\alpha B_\alpha^*A_\alpha).\]
\end{theorem}
\begin{proof}
Remark that $\eta_G(A(G))$ is dense in $L^2(G)$, where $A(G)$ is the $*$-subalgebra of $C(G)$ generated by matrix coefficients of finite dimensional unitary corepresentations. Thus, by the Schur orthogonal relation (see \cite[Theorem 1.4.2]{NeshveyevTuset}), $\{\eta_G(u_{ij}^{\alpha})\mid i,j =1,\dots, \dim H_\alpha, \alpha\in\widehat G\}$ forms an orthogonal basis for $L^2(G)$ and for any $\xi, \eta\in L^2(G)$ we have
\[\langle\xi, \eta\rangle=\sum_{\alpha\in\widehat G}d_q(\alpha)\mathrm{Tr}_{H_\alpha}(F_\alpha\hat\eta(\alpha)^*\hat\xi(\alpha)).\]
\end{proof}

Let $L^\infty(G)$ be the von Neumann algebra generated by $\pi_{h_G}(C(G))$. Recall that $L^\infty(G)_*$ is a Banach algebra and there exists a representation $(\pi_G, L^2(G))$ of $L^\infty(G)_*$ such that $\pi_G(\omega)\eta_G(a)=\eta_G((\mathrm{id}\otimes\omega)(\delta_G(a))$ for any $a\in C(G)$ and $\omega\in L^\infty(G)_*$. Then $W^*(G)$ is the von Neumann algebra generated by $\pi_G(L^\infty(G)_*)$. 

We can show that $\mathcal{F}_G\pi_G(\omega)\mathcal{F}_G^{-1}=L_G((\pi_\alpha(\omega))_{\alpha\in\widehat G})$ for any $\omega\in L^\infty(G)_*$, where $\pi_\alpha(\omega):=(\mathrm{id}\otimes\omega)(U_\alpha)$ and $L_G$ is the left multiplication of $\ell^\infty\mathchar`-\bigoplus_{\alpha\in\widehat G}B(H_\alpha)$ on $\ell^2\mathchar`-\bigoplus_{\alpha\in\widehat G}B(H_\alpha)$. Therefore, there exists a $*$-isomorphism from 
\[\Phi_G\colon W^*(G)\to \ell^\infty\mathchar`-\bigoplus_{\alpha\in\widehat G}B(H_\alpha).\] 

Then we can describe the Kac--Takesaki operator $\hat W_G$ explicitly.
\begin{corollary}\label{cor:hinata}
$V_G=\sum_{\alpha\in\widehat G}(\Phi_G^{-1}\otimes\pi_{h_G})(U_\alpha)$.
\end{corollary}
\begin{proof}
It immediately follows from Theorem \ref{thm:PW}.
\end{proof}

\section*{Acknowledgment}
The author gratefully acknowledges the useful comments from Sinji Koshida and his supervisor, Professor Yoshimichi Ueda. The author appreciates the comments on the paper \cite{BG} by Professor Alexei Borodin and Professor Vadim Gorin. This work was supported by JSPS Research Fellowship for Young Scientists (KAKENHI Grant Number JP 19J21098).

}


\end{document}